\documentclass[preprint,1p,times]{elsarticle}
\usepackage{amsmath}
\usepackage{amsfonts}
\usepackage{amssymb}
\usepackage{epstopdf}

\usepackage{color}%
\usepackage{subfigure}

\usepackage{color}%
\usepackage{subfigure}
\newtheorem{theorem}{Theorem}[section]
\newtheorem{definition}[theorem]{Definition}

\newtheorem{remark}[theorem]{Remark}

\newcounter{problem}
\newenvironment{problem}{\refstepcounter{problem}\vspace{1.5ex}
{\noindent\bf Problem \theproblem.}\hspace{0.3em}\parindent=0pt}{\vspace{1ex}}

\newcounter{example}

\newcommand{\norm}[1]{\left\Vert#1\right\Vert}

\newenvironment{proof}{\vspace{1ex}
{\it Proof. }\hspace{0.3em}}{\vspace{1ex}} \journal{ }

\begin{document}

\begin{frontmatter}
\title{Two new classes of exponential Runge--Kutta
integrators \\ for efficiently solving stiff systems or highly
oscillatory problems }


\author[1]{Bin Wang}
\ead{wangbinmaths@xjtu.edu.cn}

\author[2]{Xianfa Hu}
\ead{zzxyhxf@163.com}

\author[3]{Xinyuan Wu\corref{cor1}}
\ead{xywu@nju.edu.cn}

\cortext[cor1]{Corresponding author}
\address[1]{School of Mathematics and Statistics, Xi'an Jiaotong University, Xi'an 710049, Shannxi,
P.R.China}
\address[2]{Department of Mathematics, Shanghai Normal University, Shanghai 200234, P.R.China}
\address[3]{Department of Mathematics, Nanjing University,
Nanjing 210093, P.R.China}

\begin{abstract}
We note a fact that stiff systems or differential equations that
have highly oscillatory solutions cannot be solved efficiently using
conventional methods. In this paper, we study two new classes of
exponential Runge--Kutta (ERK) integrators for efficiently solving
stiff systems or highly oscillatory problems. We first present a
novel class of explicit modified version of  exponential
Runge--Kutta (MVERK) methods based on the order conditions.
Furthermore, we consider a class of explicit simplified version of
exponential Runge--Kutta (SVERK) methods. Numerical results
demonstrate the high efficiency of the explicit MVERK integrators
and SVERK  methods derived in this paper compared with the
well-known explicit ERK integrators for stiff systems or highly
oscillatory problems in the literature.
\end{abstract}


\end{frontmatter}

\pagestyle{myheadings} \thispagestyle{plain} \markboth{}
{\centerline{\small B. Wang, X. Hu  and X. Wu}}

\section{Introduction}\label{sec:introducton}
The advantages {of} exponential methods have been clarified in the
literature (see, e.g. \cite{Hochbruck1998,LiYW2016a}), and there is
 a MATLAB package for exponential integrators named EXPINT which
is described in \cite{Berland2007Package}. Generally speaking,
exponential methods permit larger stepsize and achieve higher
accuracy than non-exponential ones. {There is no doubt that the idea
to make use of matrix exponentials in a Runge--Kutta-type integrator
by no means new. However, most of them appeared in the literature
are expensive since the coefficients of these integrators are
heavily dependent on the evaluations of matrix exponentials, even
though an explicit exponential Runge--Kutta-type integrator of them
is also expensive when applied to the underlying systems in
practice.} Therefore, we try  to design two novel classes of
exponential Runge--Kutta methods, which can reduce the computational
cost to some extent. To this end, we first pay our attention to  a
modified version of exponential Runge--Kutta (MVERK) integrators in
this study. We then consider a simplified version of exponential
Runge--Kutta  (SVERK) methods, and this motivates the present paper.

We now consider initial value problems expressed in the autonomous
form of first-order differential equations
\begin{equation}
\label{stiffODEs} \left\{
\begin{aligned}
& y'(t)=My(t)+f(y(t)),\quad t\in [t_0,T_{end}],
\\
& y(t_0)=y_0,
\end{aligned}\right.
\end{equation}
where $y:~\mathbb{R}\rightarrow \mathbb{R}^d$,
$f:~\mathbb{R}^d\rightarrow \mathbb{R}^d$, and the matrix $(-M)$ is
a $d\times d$ symmetric positive definite or skew-Hermitian with
eigenvalues of large modulus. Problems
 of the form \eqref{stiffODEs} arise in a variety of fields in science and engineering, such as  quantum mechanics, fluid
 mechanics, flexible mechanics,
and electrodynamics. This system is frequently yielded from
linearising the integration of large stiff systems of nonlinear
initial value problems
\begin{equation*}\left\{
\begin{aligned}
z'(t)&=g(z(t)),\\
z(t_0)&=z_0.
\end{aligned}\right.
\end{equation*}
Semidiscretised mixed initial-boundary value problems of evolution
PDEs, such as advection-diffusion equations and the rewritten
Navier-Stokes equations (see, e.g. \cite{JCP(1998)_Beylkin}), also
provide examples of this type, which can be formulated as an
abstract form:
\begin{equation}\label{evolutionPDEs}
\left\{
\begin{aligned}
&\frac{\partial u(x,t)}{\partial t}=\mathcal{L}u+
\mathcal{N}(u),\quad x\in D,\ t\in[t_0,T_{end}],
\\
&B(x)u(x,t)=0, \quad x\in \partial D, \ t\ge t_0,
\\
&u(x,t_0)=g(x), \quad x\in D,
\end{aligned}
\right.
\end{equation}
where $D$ is a spatial domain with boundary $\partial D$ in
$\mathbb{R}^d$, $\mathcal{L}$ and $\mathcal{N}$ represent
respectively linear and nonlinear operators, and $B(x)$ denotes {a}
boundary operator.

The standard Runge--Kutta (RK) methods  have been deeply rooted in
researches and engineers who are interested in scientific computing
due to their simplicity and their ease of implementation.
Unfortunately, however, it is  known that stiff systems or highly
oscillatory problems cannot be solved efficiently using standard
explicit methods since standard explicit methods need a very small
stepsize and hence a long runtime to reach an acceptable accuracy.
Therefore, in the development and design of numerical algorithms,
established methods are constantly improved with the development of
 computing technology. It is true that exponential integrators (see, e.g. \cite{Berland2005,Bhatty2017,Celledoni2008,Cox2002,Dimarco2011,Dujardin2009,Hochbruck1997,Hochbruck1998,Hochbruck2005a,Hochbruck2005b,Hochbruck2010,
Kassam2005,Lawson1967,Shen2019,WangBin2021}):
\begin{equation}\label{ERK}
\left\{
\begin{aligned}
& Y_i=e^{c_ihM}y_0 + h\sum_{j=1}^{s}\bar{a}_{ij}(hM)f(Y_j), \quad
i=1,\ldots,s,
\\
& y_{1}=e^{hM}y_0+ h\sum_{i=1}^{s}\bar{b}_{i}(hM)f(Y_i),
\end{aligned}\right.
\end{equation}
that can {exactly} integrate the linear equation $y'(t)=My(t)$ are
more favorable than non-exponential integrators in solving the
problem \eqref{stiffODEs}, which exhibits remarkable `stiffness'
properties. In \eqref{ERK}, $c_i$ for $i=1,\ldots,s$ are constants,
$\bar{a}_{ij}(hM)$ and $\bar{b}_{i}(hM)$ are matrix-valued functions
of $hM$. It is worth noting that an ERK method \eqref{ERK} reduces
to a classical RK method if $M\rightarrow \mathbf{0}$. With regard
to
 semi-linear Hamiltonian systems,  symplectic
exponential methods and energy-preserving exponential methods are
important in the sense of geometric integration (see, e.g.
\cite{LiYW2016a,Mei2017,Mei-Siam2022,WangBinJCAM2019,
WangBinJCP2019}), but  all of them are implicit (normally depending
on a Newton--Raphson procedure  due to the stiffness), and hence
they are not the subject of this paper.

 The main theme of this paper is two new classes of exponential Runge--Kutta (ERK) integrators and we will focus on explicit methods. Hence, the main
contribution of this study is to present new explicit ERK methods
for solving the system \eqref{stiffODEs} with lower computational
cost. Moreover, these ERK methods will reduce the computation of
matrix exponentials as far as possible, and be more closer to the
distinguishing feature of standard Runge--Kutta methods: simplicity
and ease of implementation. 

The remainder of this paper is organised as follows. In
Section~\ref{sec:CRERKmethod}, we first formulate  a modified
version of explicit exponential   Runge--Kutta (MVERK)  methods. We
then present a simplified version of explicit exponential
Runge--Kutta (SVERK) integrators in Section~\ref{sec:SVERKmethod}.
We are concerned with the analytical aspect for our new explicit ERK
methods in Section \ref{sec:AnalysisERKmethod}.
Numerical experiments including Allen--Cahn
equation, the averaged system in wind-induced
 oscillation and the nonlinear Schr\"{o}dinger equation are
implemented in Section~\ref{numExperiments}, and the numerical
results show the comparable accuracy and efficiency of our new
explicit ERK integrators. We draw our conclusions in the last
section.

\section{A  modified version of ERK methods}\label{sec:CRERKmethod}

The idea of the underlying modified version of  ERK methods is based
on inheriting the internal stages and modifying the update of
standard RK methods.

\begin{definition} An $s$-stage modified version of ERK
(MVERK) methods applied with stepsize $h>0$ for solving
\eqref{stiffODEs} is defined by
\begin{equation}\label{MVERK}
\left\{
\begin{aligned}
& Y_i=y_0 + h\sum_{j=1}^{s}\bar{a}_{ij}(MY_j+f(Y_j)), \quad
i=1,\ldots,s,
\\
& y_{1}=e^{hM}y_0+ h\sum_{i=1}^{s}\bar{b}_{i}f(Y_i)+w_s(hM),
\end{aligned}\right.
\end{equation}
where $\bar{a}_{ij}$ and $\bar{b}_{i}$ are real constants, $w_s(hM)$
is a suitable matrix-valued function of $hM$ (or zero), and in
particular, $w_s(hM)=0$ when $M\rightarrow 0$.
\end{definition}

\begin{remark} Differently from the standard ERK methods, MVERK methods are dependent on
a matrix-valued function $w_s(hM)$.\end{remark}

\begin{remark}It is important to note that $w_s(hM)$  appearing in \eqref{MVERK} is independent of
matrix-valued exponentials and will change with the order $p$ of the
underlying MVERK method. However, the MVERK methods with same order
$p$ share the same $w_s(hM)$. The  choice of $w_s(hM)$ relies
heavily on the order conditions,  {which must coincide} with the
order conditions for the standard RK methods.

\end{remark}

It is clear that our MVERK methods exactly integrate the following
homogeneous linear system

\begin{equation}\label{homog}
y'(t)=My(t),\quad y(0)=y_0,
\end{equation}
which has exact solution
\begin{align*}
y(t)=e^{tM}y_0.
\end{align*}
This  is an essential property of an exponential integrator. Since
$(-M)$ appearing in \eqref{stiffODEs}  is symmetric positive
definite or skew-Hermitian with eigenvalues of large modulus, the
exponential $e^{tM}$ possesses many nice features such as uniform
boundedness. In particular, the exponential contains the full
information on linear oscillations {when} \eqref{stiffODEs} is a
highly oscillatory problem.

The MVERK method \eqref{MVERK} can be represented briefly in
 Butcher's notation by the following block tableau of coefficients:
\begin{equation}\label{Tableau-MVERK}
\begin{aligned} &\quad\quad\begin{tabular}{c|c|c}
 $\bar{c}$&$\mathbf{I}$&$\bar{A}$ \\
 \hline
  $\raisebox{-1.3ex}[1.0pt]{$e^{hM}$}$ & $\raisebox{-1.3ex}[1.0pt]{$w_s(hM)$}$&$\raisebox{-1.3ex}[1.0pt]{$\bar{b}^{\intercal}$}$  \\
\end{tabular}
~=
\begin{tabular}{c|c|ccc}
 $\bar{c}_1$&$I$&$\bar{a}_{11}$&$\cdots$&$\bar{a}_{1s}$\\
$\vdots$& $\vdots$ & $\vdots$&$\vdots$\\
 $\bar{c}_s$ &$I$&  $\bar{a}_{s1}$& $\cdots$& $\bar{a}_{ss}$\\
 \hline
 $\raisebox{-1.3ex}[1.0pt]{$e^{hM}$}$&$\raisebox{-1.3ex}[1.0pt]{$w_s(hM)$}$&$\raisebox{-1.3ex}[1.0pt]{$\bar{b}_1$}$&\raisebox{-1.3ex}[1.0pt]{$\cdots$} &  $\raisebox{-1.3ex}[1.0pt]{$\bar{b}_s$}$\\
\end{tabular}
\end{aligned}
\end{equation}
with $\bar{c}_i=\sum\limits_{j=1}^s\bar{a}_{ij}$. Explicitly, the
method \eqref{MVERK} utilises the matrix exponential of $M$ and a
related function, and hence its name ``exponential integrators''.
Moreover, since the computational cost of the product of a matrix
exponential function with a vector is expensive, the internal stages
of the method \eqref{MVERK} avoid matrix exponentials, namely, the
update makes use of the matrix exponential of $M$ only once at each
step, and hence its name ``modified version of exponential
integrators''.

According to the definition of MVERK methods, it is clear that MVERK
methods can be thought of as a generalization of standard RK
methods, but the most important aspect is that MVERK methods are
specially designed for efficiently solving \eqref{stiffODEs}. In
fact, when $M\rightarrow 0$, $w_s(hM)=0$, and  then the MVERK method
\eqref{MVERK} reduces to the standard RK method:
\begin{equation*}
\left\{
\begin{aligned}
& Y_i=y_0 + h\sum_{j=1}^{s}\bar{a}_{ij}f(Y_j), \quad i=1,\ldots,s,
\\
& y_{1}=y_0+ h\sum_{i=1}^{s}\bar{b}_{i}f(Y_i).
\end{aligned}\right.
\end{equation*}

In what follows, we will present some examples of explicit MVERK
methods. As the first example of explicit MVERK methods, we consider
the special case of the one-stage explicit MVERK method with
$w_1(hM)=0$:
\begin{equation}\label{MVERK-one-stage}
\left\{
\begin{aligned}
& Y_1=y_0,
\\
& y_{1}=e^{hM}y_0+ h\bar{b}_{1}f(Y_1).
\end{aligned}\right.
\end{equation}

We will compare the Taylor series of the  numerical solution $y_{1}$
with the Taylor series of the exact solution $y(h)$ under the
assumption $y(0) = y_0$. An MVERK method whose series when expanded
about $y_0$ agrees with that of the exact solution up to the term in
$h^p$ is said to be of order $p$. The series for the numerical
solution involve the same derivatives as for the exact solution but
have coefficients that depend on the method. The resulting
conditions on these coefficients are called the order conditions.

The Taylor series for the exact solution is given by
\begin{align*}
&y(h) = y(0) + hy'(0) + \frac{h^2}{2!}y''(0) + \frac{h^3}{3!}y'''(0) + \frac{h^4}{4!}y^{(4)}(0) + \cdots \\
&= y(0) + h(My(0)+f(y(0))) + \frac{h^2}{2}(My'(0)+f'_{y}(y(0))y'(0))\\&\qquad\qquad+\frac{h^3}{3!}y'''(0) + \frac{h^4}{4!}y^{(4)}(0) + \cdots\\
&= y(0)+h(My(0)+f(y(0))) +
\frac{h^2}{2}(M(My(0)+f(y(0)))\\&\qquad\qquad+f'_{y}(y(0))(My(0)+f(y(0))))+\frac{h^3}{3!}y'''(0)
+ \frac{h^4}{4!}y^{(4)}(0) + \cdots.
\end{align*}
The Taylor series for the exact solution is to be compared with the
Taylor series for the numerical solution. First, we regard the
internal stage vector $Y(h)$ as a function of $h$ and compute
derivatives with respect to $h$, getting
\begin{align*}
Y_1 &= y_0, \\
y_1 &=e^{hM}y_0+h\bar{b}_{1}f(Y_1)\\
&=e^{hM}y_0+h\bar{b}_{1}f( y_0)\\
&=(I+hM)y_0+h\bar{b}_{1}f(y_0)+\mathcal{O}(h^2),\\
y(h)&=y_0+hy'(0)+\mathcal{O}(h^2)=y_0+h(My_0+f(y_0))+\mathcal{O}(h^2)
\end{align*}
If we consider the underling one-stage MVERK method is of order one,
we then obtain $\bar{b}_1=1$. This gives the following first-order
explicit MVERK method with one stage
\begin{equation}\label{MVERK-one-stage-Euler}
y_1=e^{hM}y_0+hf(y_0),
\end{equation}
which can be expressed in the Butcher tableau
\begin{equation}\label{}
\begin{aligned}
\begin{tabular}{c|c|c}
 $0$&$I$&$0$\\
\hline
  $\raisebox{-1.3ex}[1.0pt]{$e^{hM}$}$&$\raisebox{-1.3ex}[1.0pt]{$0$}$&$\raisebox{-1.3ex}[1.0pt]{$1$}$\\
\end{tabular}~.
\end{aligned}
\end{equation}
The first-order explicit MVERK method with one stage is also termed
the modified exponential Euler method, which is different from the
exponential Euler method as it stands (see, e.g. Acta Numer. (2010)
by Hochbruck et al.)
\begin{equation}\label{Pro-exponential-Euler}
\begin{aligned}
y_1=e^{hM}y_0+h\varphi_1(hM)f(y_0),
\end{aligned}
\end{equation}
where
\begin{equation}\label{varphi1}
\begin{aligned}
\varphi_1(z)&=\frac{e^z-1}{z}.
\end{aligned}
\end{equation}
They would be same if $\varphi_1(z)$ were replaced by $1$. Here, it
is worth mentioning that the exponential Euler method
\eqref{Pro-exponential-Euler} is very popular, which is the
prototype exponential method appeared repeatedly in the literature.

It is noted that when $M\rightarrow 0$, the modified  exponential
Euler method \eqref{MVERK-one-stage-Euler} reduces to the well-known
explicit Euler method for $y'=f(y)$.

As the second example of MVERK methods, we then consider two-stage
explicit MVERK methods with $w_2(hM)=\frac{h^2}{2}Mf(y_0)$:
\begin{equation}\label{MVERK-two-stage}
\left\{
\begin{aligned}
& Y_1=y_0, \\
& Y_2=y_0+h\bar{a}_{21}(My_0+f(y_0)),
\\
& y_{1}=e^{hM}y_0+
h(\bar{b}_{1}f(Y_1)+\bar{b}_{2}f(Y_2))+\frac{h^2}{2}Mf(y_0).
\end{aligned}\right.
\end{equation}
Considering the second-order explicit MVERK method with two stages
yields
\begin{equation}\label{MVERK-two-OC}
\begin{aligned}
\includegraphics{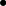}\qquad\qquad&\bar{b}_1+\bar{b}_2&=1, \\
\includegraphics{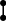}\qquad\qquad&2\bar{a}_{21}\bar{b}_2&=1.\\
\end{aligned}
\end{equation}

\noindent Let $\bar{a}_{21}\neq 0$ be real parameter. We obtain a
one-parameter family of second-order explicit MVERK method:

\begin{align*}
\bar{b}_2&=\frac{1}{2\bar{a}_{21}},\ \ \ \
\bar{b}_1=1-\frac{1}{2\bar{a}_{21}},\quad \bar{a}_{21}\neq 0.
\end{align*}
The choice of $\bar{a}_{21}=1$ gives
$\bar{b}_1=\bar{b}_2=\frac{1}{2}$. This suggests the following
second-order explicit MVERK method with two stages
\begin{equation}\label{MVERK-two-stage-1}
\left\{
\begin{aligned}
& Y_1=y_0, \\
& Y_2=y_0+h(My_0+f(y_0)),
\\
& y_{1}=e^{hM}y_0+ \frac{h}{2}((I+Mh)f(Y_1)+f(Y_2)),
\end{aligned}\right.
\end{equation}
which can be denoted by the Butcher tableau

\begin{equation}\label{}
\begin{aligned}
\begin{tabular}{c|c|cc}
 $0$&$I$&$0$&$0$\\
$1$ &$I$&  $1$& $0$\\
 \hline
 $\raisebox{-1.3ex}[1.0pt]{$e^{hM}$}$ &$\raisebox{-1.3ex}[1.0pt]{$w_2(hM)$}$&$\raisebox{-1.3ex}[1.0pt]{$\frac{1}{2}$}$&  $\raisebox{-1.3ex}[1.0pt]{$\frac{1}{2}$}$\\
\end{tabular}~.
\end{aligned}
\end{equation}

The choice of $\bar{a}_{21}=\frac{1}{2}$ delivers $\bar{b}_1=0$ and
$\bar{b}_2=1$. This leads to another second-order explicit MVERK
method with two stages
\begin{equation}\label{MVERK-two-stage-2}
\left\{
\begin{aligned}
& Y_1=y_0, \\
& Y_2=y_0+\frac{h}{2}(My_0+f(y_0)),
\\
& y_{1}=e^{hM}y_0+ h(f(Y_2)+\frac{h}{2}Mf(Y_1)),
\end{aligned}\right.
\end{equation}
which can be presented by the Butcher tableau
\begin{equation}\label{}
\begin{aligned}
\begin{tabular}{c|c|cc}
 $0$&$I$&$0$&$0$\\
$\frac{1}{2}$ &$I$&  $\frac{1}{2}$& $0$\\
 \hline
 $\raisebox{-1.3ex}[1.0pt]{$e^{hM}$}$ &$\raisebox{-1.3ex}[1.0pt]{$w_2(hM)$}$&$\raisebox{-1.3ex}[1.0pt]{$0$}$&  $\raisebox{-1.3ex}[1.0pt]{$1$}$\\
\end{tabular}~.
\end{aligned}
\end{equation}

It is noted that when  $M\rightarrow 0$, the second-order explicit
MVERK methods \eqref{MVERK-two-stage-1} and
\eqref{MVERK-two-stage-2}
 with two stages reduce to the second-order RK method by Heun,  and the so-called modified Euler method by Runge, respectively.

In what follows we  consider three-stage explicit MVERK methods of
order three with
$w_3(hM)=\frac{1}{6}h^2M(3f(y_0)+h(Mf(y_0)+f'_y(y_0)(My_0+f(y_0)))):$

\begin{equation}\label{MVERK-three-stage}
\left\{
\begin{aligned}
& Y_1=y_0, \\
& Y_2=y_0+h\bar{a}_{21}(My_0+f(y_0)),\\
& Y_3=y_0+h(\bar{a}_{31}(MY_1+f(Y_1))+\bar{a}_{32}(MY_2+f(Y_2)))\\
& y_{1}=e^{hM}y_0+
h(\bar{b}_1f(Y_1)+\bar{b}_{2}f(Y_2)+\bar{b}_3f(Y_3))\\
&\qquad+\frac{1}{6}h^2M(3f(y_0)+h(Mf(y_0)+f'_y(y_0)(My_0+f(y_0)))).
\end{aligned}\right.
\end{equation}
The order conditions for the three-order explicit MVERK methods with
three stages are given by
\begin{equation}\label{MVERK-three-OC}
\begin{aligned}
\includegraphics{t1.eps}\qquad\qquad&\bar{b}_1+\bar{b}_2+\bar{b}_3&=1,\\
\includegraphics{t2.eps}\qquad\qquad&\bar{b}_2\bar{a}_{21}+\bar{b}_3(\bar{a}_{31}+\bar{a}_{32})&=\frac{1}{2},\\
\includegraphics{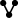}\qquad\qquad&\bar{b}_2\bar{a}_{21}^2+\bar{b}_3(\bar{a}_{31}+\bar{a}_{32})^2&=\frac{1}{3},\\
\includegraphics{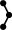}\qquad\qquad&\bar{b}_3\bar{a}_{32}\bar{a}_{21}&=\frac{1}{6}.\\
\end{aligned}
\end{equation}
The system \eqref{MVERK-three-OC} has infinitely many solutions, and
we refer the reader to Chap. 3 in \cite{Butcherbook2008} for
details. It can be verified that $\bar{b}_2=\bar{a}_{31}=0,\
\bar{a}_{21}=\frac{1}{3},\ \bar{a}_{32}=\frac{2}{3}$,
$\bar{b}_1=\frac{1}{4}$, and $\bar{b}_3=\frac{3}{4}$ satisfy the
order conditions stated above. Accordingly, we obtain the
three-order explicit MVERK method with three stages as follows:

\begin{equation}\label{MVERK-three-stage-1}
\left\{
\begin{aligned}
& Y_1=y_0, \\
& Y_2=y_0+\frac{1}{3}hg_0,\\
& Y_3=y_0+\frac{2}{3}h(MY_2+f(Y_2))\\
& y_{1}=e^{hM}y_0+ \frac{1}{4}h(f(Y_1)+3f(Y3))\\&\qquad\qquad\ \
+\frac{1}{6}h^2M(3f(y_0)+h(Mf(y_0)+f'_y(y_0)g_0)),
\end{aligned}\right.
\end{equation}
where
\begin{equation*}
\begin{aligned}
g_0&=My_0+f(y_0).
\end{aligned}
\end{equation*}
The MVERK method \eqref{MVERK-three-stage-1} can be expressed in the
Butcher tableau
\begin{equation}\label{}
\begin{aligned}
\begin{tabular}{c|c|ccc}
$0$ &$I$&$0$&&\\
$\raisebox{-1.3ex}[1.0pt]{$\frac{1}{3}$}$& $\raisebox{-1.3ex}[1.0pt]{$I$}$ &$\raisebox{-1.3ex}[1.0pt]{$\frac{1}{3}$}$ &$\raisebox{-1.3ex}[1.0pt]{$0$}$\\
 $\raisebox{-1.3ex}[1.0pt]{$\frac{2}{3}$}$ &$\raisebox{-1.3ex}[1.0pt]{$I$}$&  $\raisebox{-1.3ex}[1.0pt]{$0$}$& $\raisebox{-1.3ex}[1.0pt]{$\frac{2}{3}$}$& $\raisebox{-1.3ex}[1.0pt]{$0$}$\\
 \hline
 $\raisebox{-1.3ex}[1.0pt]{$e^{hM}$}$ &$\raisebox{-1.3ex}[1.0pt]{$w_3(hM)$}$&$\raisebox{-1.3ex}[1.0pt]{$\frac{1}{4}$}$&\raisebox{-1.3ex}[1.0pt]{$0$} &  $\raisebox{-1.3ex}[1.0pt]{$\frac{3}{4}$}$\\
\end{tabular}~.
\end{aligned}
\end{equation}
When $M\rightarrow 0$, the MVERK method \eqref{MVERK-three-stage-1}
reduces to Heun's three-order RK method. It is true that
\eqref{MVERK-three-stage-1} uses the Jacobian matrix of $f(y)$ with
respect to $y$ at each step. However, as is known, an A-stable RK
method is implicit, and  the Newton--Raphson iteration is required
when applied to stiff systems. This implies that an A-stable RK
method when applied to stiff systems depends on the evaluation of
Jacobian matrix of $f(y)$ with respect to $y$ at each step as well.
As we have emphasised in Introduction that implicit exponential
integrators need to use the evaluation of Jacobian matrix of $f(y)$
with respect to $y$, depending on an iterative procedure due to the
stiffness of \eqref{stiffODEs}. On the other hand, the computational
cost of explicit exponential integrators appeared in the literature
depends on evaluations of matrix exponentials heavily. If the cost
of computing the Jacobian matrix of $f(y)$ with respect to $y$ for
the underlying system \eqref{stiffODEs} is cheaper than that of the
evaluations of matrix exponentials for the explicit exponential
integrators appeared in the literature, we are also hopeful of
obtaining the high efficiency of our three-order explicit
exponential integrators with three stages.

 Also, it can be verified that $\bar{a}_{21}=\frac{1}{2},\ \bar{a}_{31}=0,\ \bar{a}_{32}=\frac{3}{4}$, $\bar{b}_1=\frac{2}{9},\ \bar{b}_2=
\frac{3}{9}$ and $\bar{b}_3=\frac{4}{9}$ satisfy the order
conditions. Consequently, we obtain another three-order explicit
MVERK method with three stages as follows:

\begin{equation}\label{MVERK-three-stage-2}
\left\{
\begin{aligned}
& Y_1=y_0, \\
& Y_2=y_0+\frac{1}{2}hg_0,\\
& Y_3=y_0+\frac{3}{4}h(MY_2+f(Y_2))\\
& y_{1}=e^{hM}y_0+
\frac{1}{9}h(2f(Y_1)+3f(Y_2)+4f(Y3))\\&\qquad\qquad\ \
+\frac{1}{6}h^2M(3f(y_0)+h(Mf(y_0)+f'_y(y_0)g_0)),
\end{aligned}\right.
\end{equation}
where
\begin{equation*}
\begin{aligned}
g_0&=My_0+f(y_0).
\end{aligned}
\end{equation*}
The MVERK method \eqref{MVERK-three-stage-2} can be denoted by the
Butcher tableau
\begin{equation}\label{}
\begin{aligned}
\begin{tabular}{c|c|ccc}
$0$ &$I$&$0$&&\\
$\raisebox{-1.3ex}[1.0pt]{$\frac{1}{2}$}$& $\raisebox{-1.3ex}[1.0pt]{$I$}$ &$\raisebox{-1.3ex}[1.0pt]{$\frac{1}{2}$}$ &$\raisebox{-1.3ex}[1.0pt]{$0$}$\\
 $\raisebox{-1.3ex}[1.0pt]{$\frac{3}{4}$}$ &$\raisebox{-1.3ex}[1.0pt]{$I$}$&  $\raisebox{-1.3ex}[1.0pt]{$0$}$& $\raisebox{-1.3ex}[1.0pt]{$\frac{3}{4}$}$& $\raisebox{-1.3ex}[1.0pt]{$0$}$\\
 \hline
  $\raisebox{-1.3ex}[1.0pt]{$e^{hM}$}$&$\raisebox{-1.3ex}[1.0pt]{$w_3(hM)$}$&$\raisebox{-1.3ex}[1.0pt]{$\frac{2}{9}$}$&\raisebox{-1.3ex}[1.0pt]{$\frac{3}{9}$} &  $\raisebox{-1.3ex}[1.0pt]{$\frac{3}{4}$}$\\
\end{tabular}
\end{aligned}~.
\end{equation}

When $M\rightarrow 0$, the MVERK method \eqref{MVERK-three-stage-2}
reduces to the classical third-order RK method with three stages.

\section{A simplified version of ERK methods}\label{sec:SVERKmethod}
Following the idea stated in the previous section, in this section
we will consider a simplified version of ERK methods, which also
allows internal stages to use matrix exponentials of $hM$ to some
extent.

Differently from MVERK methods, here the internal stages make use of
matrix exponentials of $hM$, but all the coefficients of the
simplified version are independent of matrix
exponentials.\vskip0.3cm

\begin{definition} An $s$-stage simplified version of ERK (SVERK) method
applied with stepsize $h>0$ for solving \eqref{stiffODEs} is defined
by
\begin{equation}\label{SVERK}
\left\{
\begin{aligned}
& Y_i=e^{\bar{c}_ihM}y_0 + h\sum_{j=1}^{s}\bar{a}_{ij}f(Y_j), \quad
i=1,\ldots,s,
\\
& y_{1}=e^{hM}y_0+
h\sum_{i=1}^{s}\bar{b}_{i}f(Y_i)+\widetilde{w}_s(hM),
\end{aligned}\right.
\end{equation}
where $\bar{a}_{ij}$ and $\bar{b}_{i}$ are real constants,
$\bar{c}_i=\sum\limits_{j=1}^s\bar{a}_{ij}$, $\widetilde{w}_s(hM)$
is a suitable matrix-valued function of $hM$ (or zero), and in
particular, $\widetilde{w}_s(hM)=0$ when $M\rightarrow 0$.
\end{definition}
Likewise, differently from the standard ERK methods, SVERK methods
are dependent on a matrix-valued function $\widetilde{w}_s(hM)$.
Here $\widetilde{w}_s(hM)$  appearing in \eqref{SVERK} is
independent of matrix-valued exponentials, and will change with the
order $p$ of the underlying SVERK method. However, the SVERK methods
with same order $p$ share the same $\widetilde{w}_s(hM)$. The choice
of $\widetilde{w}_s(hM)$ relies heavily on the order conditions,
 {which must} coincide with the order conditions for the standard
RK methods. Obviously, our SVERK methods exactly integrate
\eqref{homog} as well.

The method \eqref{SVERK} can be represented briefly in
 Butcher's notation by the following block tableau of coefficients:
\begin{equation}\label{Tableau-SVERK}
\begin{aligned} &\quad\quad\begin{tabular}{c|c|c}
 $\bar{c}$&$\mathbf{e}^{\bar{c}hM}$&$\bar{A}$ \\
 \hline
  $\raisebox{-1.3ex}[1.0pt]{$e^{hM}$}$ & $\raisebox{-1.3ex}[1.0pt]{$\widetilde{w}_s(hM)$}$&$\raisebox{-1.3ex}[1.0pt]{$\bar{b}^{\intercal}$}$  \\
\end{tabular}~=
\begin{tabular}{c|c|ccc}
 $\bar{c}_1$&$e^{\bar{c_1}hM}$&$\bar{a}_{11}$&$\cdots$&$\bar{a}_{1s}$\\
$\vdots$& $\vdots$ & $\vdots$&$\vdots$\\
 $\bar{c}_s$ &$e^{\bar{c}_shM}$&  $\bar{a}_{s1}$& $\cdots$& $\bar{a}_{ss}$\\
 \hline
 $\raisebox{-1.3ex}[1.0pt]{$e^{hM}$}$ &$\raisebox{-1.3ex}[1.0pt]{$\widetilde{w}_s(hM)$}$&$\raisebox{-1.3ex}[1.0pt]{$\bar{b}_1$}$&\raisebox{-1.3ex}[1.0pt]{$\cdots$} &  $\raisebox{-1.3ex}[1.0pt]{$\bar{b}_s$}$\\
\end{tabular}~.
\end{aligned}
\end{equation}

We first consider the explicit one-stage SVERK method of order one
with $\widetilde{w}_1(hM)=0$. It is easy to see that $\bar{b}_1=1$.
This implies the following first-order explicit SVERK method with
one stage:
\begin{equation}\label{SVERK-one-stage-Euler}
y_1=e^{hM}y_0+hf(y_0),
\end{equation}
which is identical to \eqref{MVERK-one-stage-Euler}.

We next consider second-order explicit SVERK methods with
$\widetilde{w}_2(hM)=\frac{h^2}{2}Mf(y_0)$:
\begin{equation}\label{SVERK-two-stage}
\left\{
\begin{aligned}
& Y_1=y_0, \\
& Y_2=e^{\bar{c}_2hM}y_0+h\bar{a}_{21}f(y_0),
\\
& y_{1}=e^{hM}y_0+
h(\bar{b}_{1}f(Y_1)+\bar{b}_{2}f(Y_2))+\frac{h^2}{2}Mf(y_0).
\end{aligned}\right.
\end{equation}
Considering the second-order SVERK method with two stages yields
\begin{align*}
\bar{c}_2&=\bar{a}_{21},\ \ \ \
\bar{b}_1+\bar{b}_2=1,\ \ \ \
2\bar{a}_{21}\bar{b}_2=1.
\end{align*}
This is the same as \eqref{MVERK-two-OC} of the order conditions for
MVERK methods. Let $\bar{a}_{21}\neq 0$ be real parameter. We obtain
a one-parameter family of second-order explicit SVERK method:
\begin{align*}
\bar{b}_2&=\frac{1}{2\bar{a}_{21}},\ \ \ \
\bar{b}_1=1-\frac{1}{2\bar{a}_{21}},\quad \bar{a}_{21}\neq 0.
\end{align*}
The choice of $\bar{a}_{21}=1$ gets
$\bar{b}_1=\bar{b}_2=\frac{1}{2}$. This results in the following
second-order SVERK method with two stages:

\begin{equation}\label{SVERK-two-stage-1}
\left\{
\begin{aligned}
& Y_1=y_0, \\
& Y_2=e^{hM}y_0+hf(y_0),
\\
& y_{1}=e^{hM}y_0+ \frac{h}{2}(f(Y_1)+f(Y_2))+\frac{h^2}{2}Mf(y_0),
\end{aligned}\right.
\end{equation}
which can be denoted by the Butcher tableau

\begin{equation}\label{}
\begin{aligned}
\begin{tabular}{c|c|cc}
 $0$&$I$&$0$&$0$\\
$1$ &$e^{hM}$&  $1$& $0$\\
 \hline
 $\raisebox{-1.3ex}[1.0pt]{$e^{hM}$}$ &$\raisebox{-1.3ex}[1.0pt]{$\widetilde{w}_2(hM)$}$&$\raisebox{-1.3ex}[1.0pt]{$\frac{1}{2}$}$&  $\raisebox{-1.3ex}[1.0pt]{$\frac{1}{2}$}$\\
\end{tabular}~.
\end{aligned}
\end{equation}

The choice of $\bar{a}_{21}=\frac{1}{2}$ yields $\bar{b}_1=0$ and
$\bar{b}_2=1$. This arrives at another second-order explicit SVERK
method with two stages
\begin{equation}\label{SVERK-two-stage-2}
\left\{
\begin{aligned}
& Y_1=y_0, \\
& Y_2=e^{\frac{h}{2}M}y_0+\frac{h}{2}f(y_0),
\\
& y_{1}=e^{hM}y_0+
hf(Y_2)+\frac{h^2}{2}Mf(y_0),\\
\end{aligned}\right.
\end{equation}
which can be expressed in the Butcher tableau

\begin{equation}\label{}
\begin{aligned}
\begin{tabular}{c|c|cc}
 $0$&$I$&$0$&$0$\\
$\frac{1}{2}$ &$e^{\frac{h}{2}M}$&  $\frac{1}{2}$& $0$\\
 \hline
  $\raisebox{-1.3ex}[1.0pt]{$e^{hM}$}$&$\raisebox{-1.3ex}[1.0pt]{$\widetilde{w}_2(hM)$}$&$\raisebox{-1.3ex}[1.0pt]{$0$}$&  $\raisebox{-1.3ex}[1.0pt]{$1$}$\\
\end{tabular}~.
\end{aligned}
\end{equation} It is noted that when $M\rightarrow 0$ and
$hM\rightarrow 0$, the 2-stage SVERK methods
\eqref{SVERK-two-stage-1} and \eqref{SVERK-two-stage-2}
 of order two reduce to the well-known
explicit RK method of order two by Heun and the so-called modified
Euler method, respectively.

Likewise, we derive three-stage explicit SVERK methods with $
\widetilde{w}_3(hM)=\frac{h^2}{2}Mf(y_0)+\frac{1}{6}h^3((M+f'_y(y_0))Mf(y_0)+Mf'_y(y_0)(My_0+f(y_0))):
$
\begin{equation}\label{SVERK-three-stage}
\left\{
\begin{aligned}
& Y_1=y_0, \\
& Y_2=e^{\bar{c}_2hM}y_0+h\bar{a}_{21}f(Y_1),
\\
& Y_3=e^{\bar{c}_3hM}y_0+h(\bar{a}_{31}f(Y_1)+\bar{a}_{32}f(Y_2)),\\
& y_{1}=e^{hM}y_0+
h(\bar{b}_{1}f(Y_1)+\bar{b}_{2}f(Y_2)+\bar{b}_3f(Y_3))+\frac{h^2}{2}Mf(y_0)\\&\qquad+\frac{1}{6}h^3((M+f'_y(y_0))Mf(y_0)+Mf'_y(y_0)(My_0+f(y_0))).
\end{aligned}\right.
\end{equation}
The order conditions for 3-stage SVERK methods of order three are
the same as \eqref{MVERK-three-OC}.\vskip0.5cm

We choose $\bar{a}_{21}=\frac{1}{2},\ \bar{a}_{31}=0,\
\bar{a}_{32}=\frac{3}{4},\ \bar{b}_{1}=\frac{2}{9},\
\bar{b}_{2}=\frac{3}{9}$ and $\bar{b}_{3}=\frac{4}{9}$. It can be
verified that this choice satisfies the order conditions
\eqref{MVERK-three-OC}. Hence, we have following third-order
explicit SVERK methods with three stages:

\begin{equation}\label{SVERK-three-stage-1}
\left\{
\begin{aligned}
& Y_1=y_0, \\
& Y_2=e^{\frac{h}{2}M}y_0+\frac{h}{2}f(y_0),
\\
& Y_3=e^{\frac{3}{4}hM}y_0+\frac{3}{4}hf(Y_2),\\
& y_{1}=e^{hM}y_0+
\frac{h}{9}(2f(Y_1)+3f(Y_2)+4f(Y_3))+\frac{h^2}{2}Mf(y_0)\\&\qquad+\frac{1}{6}h^3((M+f'_y(y_0))Mf(y_0)+Mf'_y(y_0)(My_0+f(y_0))).
\end{aligned}\right.
\end{equation}
The SVERK method \eqref{SVERK-three-stage-1} can be expressed in the
Butcher tableau
\begin{equation}\label{}
\begin{aligned}
\begin{tabular}{c|c|ccc}
$0$ &$I$&$0$&&\\
$\raisebox{-1.3ex}[1.0pt]{$\frac{1}{2}$}$& $\raisebox{-1.3ex}[1.0pt]{$e^{\frac{h}{2}M}$}$ &$\raisebox{-1.3ex}[1.0pt]{$\frac{1}{2}$}$ &$\raisebox{-1.3ex}[1.0pt]{$0$}$\\
$\raisebox{-1.3ex}[1.0pt]{ $\frac{3}{4}$}$ &$\raisebox{-1.3ex}[1.0pt]{$e^{\frac{3}{4}hM}$}$&  $\raisebox{-1.3ex}[1.0pt]{$0$}$& $\raisebox{-1.3ex}[1.0pt]{$\frac{3}{4}$}$& $\raisebox{-1.3ex}[1.0pt]{$0$}$\\
 \hline
  $\raisebox{-1.3ex}[1.0pt]{$e^{hM}$}$&$\raisebox{-1.3ex}[1.0pt]{$\widetilde{w}_3(hM)$}$&$\raisebox{-1.3ex}[1.0pt]{$\frac{2}{9}$}$&\raisebox{-1.3ex}[1.0pt]{$\frac{3}{9}$} &  $\raisebox{-1.3ex}[1.0pt]{$\frac{4}{9}$}$\\
\end{tabular}~.
\end{aligned}
\end{equation}
Another option is that $\bar{a}_{21}=\frac{1}{3},\ \bar{a}_{31}=0,\
\bar{a}_{32}=\frac{2}{3},\ \bar{b}_{1}=\frac{1}{4},\ \bar{b}_{2}=0$
and $\bar{b}_{3}=\frac{3}{4}$. It is easy to see that this choice
satisfies the order conditions \eqref{MVERK-three-OC}. Thus, we
obtain the third-order explicit SVERK methods with three stages as
follows:

\begin{equation}\label{SVERK-three-stage-2}
\left\{
\begin{aligned}
& Y_1=y_0, \\
& Y_2=e^{\frac{1}{3}hM}y_0+\frac{1}{3}hf(y_0),
\\
& Y_3=e^{\frac{2}{3}hM}y_0+\frac{2}{3}hf(Y_2),\\
& y_{1}=e^{hM}y_0+
\frac{h}{4}(f(Y_1)+3f(Y_3))+\frac{h^2}{2}Mf(y_0)\\&\qquad+\frac{1}{6}h^3((M+f'_y(y_0))Mf(y_0)+Mf'_y(y_0)(My_0+f(y_0))),
\end{aligned}\right.
\end{equation}
which can be denoted by the Butcher tableau
\begin{equation}\label{}
\begin{aligned}
\begin{tabular}{c|c|ccc}
$0$ &$I$&$0$&&\\
$\raisebox{-1.3ex}[1.0pt]{$\frac{1}{3}$}$& $\raisebox{-1.3ex}[1.0pt]{$e^{\frac{h}{3}M}$}$ &$\raisebox{-1.3ex}[1.0pt]{$\frac{1}{3}$}$ &$\raisebox{-1.3ex}[1.0pt]{$0$}$\\
$\raisebox{-1.3ex}[1.0pt]{ $\frac{2}{3}$}$ &$\raisebox{-1.3ex}[1.0pt]{$e^{\frac{2}{3}hM}$}$&  $\raisebox{-1.3ex}[1.0pt]{$0$}$& $\raisebox{-1.3ex}[1.0pt]{$\frac{2}{3}$}$&$\raisebox{-1.3ex}[1.0pt]{ $0$}$\\
 \hline
 $\raisebox{-1.3ex}[1.0pt]{$e^{hM}$}$ &$\raisebox{-1.3ex}[1.0pt]{$\widetilde{w}_3(hM)$}$&$\raisebox{-1.3ex}[1.0pt]{$\frac{1}{4}$}$&\raisebox{-1.3ex}[1.0pt]{$0$} &  $\raisebox{-1.3ex}[1.0pt]{$\frac{3}{4}$}$\\
\end{tabular}~.
\end{aligned}
\end{equation}
When $M\rightarrow 0$, the third-order explicit  SVERK methods
\eqref{SVERK-three-stage-1} and \eqref{SVERK-three-stage-2}
 with three stages reduce to the classical
explicit third-order RK method   and the well-known Heun's method of
order three, respectively.

\section{Analysis issues}\label{sec:AnalysisERKmethod}

In this section, we aim at some analytical aspects associated with
our new explicit ERK methods.

\begin{theorem}\label{A-Stab} All the new exponential integrators
presented in this paper are $A$-stable in the sense of Dahlquist
(see \cite{DahlquistBIT1963}).
\end{theorem}
\begin{proof}
It is easy to see that when applied to the initial value problem
$y'=\lambda y, \quad y(0)=y_0$, these new exponential integrators
with stepsize $h$ generate the approximate solution
$y_n=(e^{h\lambda})^{n}y_0$. If $\lambda$ is a complex scalar with
negative real part, then $\lim\limits_{n\rightarrow\infty}y_n=0$,
establishing $A$-stability in the sense of Dahlquist. \hfill
$\square$
\end{proof}
\begin{theorem}\label{A-Stab-23} If
$\bar{c}=(\bar{c}_1,\cdots,\bar{c}_s),~\bar{b}=(\bar{b_1},\cdots,b_s)$
and $\bar{A}=(\bar{a}_{ij})$ for $s=1,2,3$ are coefficients of a
standard explicit RK method of order $s$, then the explicit MVERK
and SVERKN methods with the same nodes $\bar{c}$, weights $\bar{b}$
and coefficients $\bar{A}$ are also of order $s$ when applied to
\eqref{stiffODEs}. Moreover, all of them are $A$-stable in the sense
of Dahlquist.
\end{theorem}
\begin{proof}
The conclusion of this theorem follows from Theorem \ref{A-Stab} and
the derivation for explicit MVERK and SVERKN methods in Section
\ref{sec:CRERKmethod} and Section \ref{sec:SVERKmethod}. \hfill
$\square$
\end{proof}

We next analyse the convergence of the first-order explicit MVERK
method \eqref{MVERK-one-stage-Euler} and the following theorem
states the corresponding result. Our analysis below will be based on
an abstract formulation of \eqref{stiffODEs} as an evolution
equation in a Banach space $(X, \norm{\cdot})$. We choose
\begin{equation}\label{alp}0 \leq \alpha< 1\end{equation}  and
define $ V = D(\tilde{M}^{\alpha}) \subset X $, where $\tilde{M}$
denotes the shifted operator $\tilde{M}=M+\omega I$ with $\omega >
-\alpha$ {and  $D(\tilde{M}^{\alpha})$ stands for the domain of
$\tilde{M}^{\alpha}$ in $X$.} The linear space $V$ is a Banach space
with norm $\norm{v}_{V}=\norm{\tilde{M}^{\alpha} v}$.
\begin{theorem}
It is assumed that \eqref{stiffODEs} has sufficiently smooth
solutions $y: [0, T] \rightarrow V $ with derivatives in $V$, and $
f:  V \rightarrow X$ is twice differentiable and $
\tilde{M}^{\gamma-1 }f^{(r)} \in L^{\infty}(0,T;V)$ with $ 0< \gamma
\leq 1$\footnote{{It is noted that for this $\gamma$ and the
$\alpha$ introduced in \eqref{alp}, this is no relation between
them.}}
 for $ r=0,1,2$.  Furthermore, let $ f $ be locally
Lipschitz-continuous, i.e., there exists a constant {$ L(R)>0$  such
that $\norm{f(y)-f(\tilde{y})}  \leq L \norm{ y- \tilde{y} }_{V}$
for all $\max(\norm{ y }_{V},\norm{ \tilde{y} }_{V})\leq R$}. Then
the convergence of the method \eqref{MVERK-one-stage-Euler} is given
by
$$\norm{e_{n}}_{V}\leq Ch,$$
where {$e_n=y_n-y(t_n)$, and} the constant $C$ is dependent on $T$,
but independent of $n$ and $h$.
\end{theorem}

\begin{proof}
Inserting the exact solution into the method
\eqref{MVERK-one-stage-Euler}, we obtain
\begin{equation}\label{e1}
y(t_{n+1})=e^{hM}y(t_n)+h\hat{f}(t_n)+\delta_{n+1},
\end{equation}
where $\delta_{n+1}$ presents the discrepancies of the method
\eqref{MVERK-one-stage-Euler}, and $\hat{f}(t)=f(y(t))$. It follows
from the variation-of-constants formula and Taylor series that
\begin{equation}\label{e2}
\begin{aligned}
y(t_{n}+h)&=e^{hM}y(t_n)+\int_{0}^{h}e^{(h-\tau)M}\hat{f}(t_n+\tau)d\tau\\
&=e^{hM}y(t_n)+\int_{0}^{h}e^{(h-\tau)M}\hat{f}(t_n)d\tau
+\int_{0}^{h}e^{(h-\tau)M}\int_{0}^{\tau}\hat{f}'(t_n+\sigma)d\sigma d\tau\\
&=e^{hM}y(t_n)+h\varphi_1(hM)\hat{f}(t_n)
+\int_{0}^{h}e^{(h-\tau)M}\int_{0}^{\tau}\hat{f}'(t_n+\sigma)d\sigma d\tau.
\end{aligned}
\end{equation}
Subtracting \eqref{e1} from \eqref{e2}, we have
\begin{equation}\label{e3}
\begin{aligned}
\delta_{n+1}
=h(\varphi_1(hM)-I)\hat{f}(t_n)+\int_{0}^{h}e^{(h-\tau)M}\int_{0}^{\tau}\hat{f}'(t_n+\sigma)d\sigma d\tau.
\end{aligned}
\end{equation}

Let $e_n=y_n-y(t_n)$, and from \eqref{MVERK-one-stage-Euler} and
\eqref{e1}, it follows that
\begin{equation*}\label{e4}
e_{n+1}=e^{hM}e_{n}+h(f(y_n)-\hat{f}(t_n))-\delta_{n+1}.
\end{equation*}
Using recursion formula, one has
\begin{equation*}\label{e5}
\sum_{j=0}^{n-1}\big(e^{(n-j-1)hM}e_{j+1}-e^{(n-j)}hMe_{j}\big)=h\sum_{j=0}^{n-1}e^{(n-j-1)hM}
(f(y_j)-\hat{f}(t_j))-\sum_{j=0}^{n-1}e^{(n-j-1)hM}\delta_{j+1}.
\end{equation*}
Then, we have
\begin{equation*}\label{e6}
e_{n}=h\sum_{j=0}^{n-1}e^{(n-j-1)hM}(f(y_j)-\hat{f}(t_j))-\sum_{j=0}^{n-1}e^{jhM}\delta_{n-j}.
\end{equation*}
We estimate the global error $e_n$ in the norm $\norm{\cdot}_{V}$ by
\begin{equation}\label{e7}
\norm{e_{n}}_{V} \leq h\norm{\sum_{j=0}^{n-1}e^{(n-j-1)hM}(f(y_j)-\hat{f}(t_j))}_{V} +\norm{\sum_{j=0}^{n-1}e^{jhM}\delta_{n-j}}_{V}.
\end{equation}
It is easily deduced from $\norm{v}_{V}=\norm{\tilde{M}^{\alpha}v}$
and the  Lemma 3.1 of \cite{Hochbruck2005a} that
\begin{equation*}\label{e8}
\begin{aligned}
& h\norm{\sum_{j=0}^{n-1}e^{(n-j-1)hM}(f(y_j)-\hat{f}(t_j))}_{V}
= h\norm{\sum_{j=0}^{n-1}\tilde{M}^{\alpha}e^{(n-j-1)hM}(f(y_j)-\hat{f}(t_j))}\\
\leq &h\sum_{j=0}^{n-1}t_{n-j-1}^{\alpha}\norm{t_{n-j-1}^{-\alpha}\tilde{M}^{\alpha}e^{(n-j-1)hM}}\norm{f(y_j)-\hat{f}(t_j)}
\leq CLh\sum_{j=0}^{n-1}t_{n-j-1}^{\alpha}\norm{e_{j}}_{V}.
\end{aligned}
\end{equation*}
Inserting the formula \eqref{e3} into the second term  on the right-hand side of \eqref{e7} gives
\begin{equation}\label{e9}
\begin{aligned}
&\norm{\sum_{j=0}^{n-1}e^{jhM}\delta_{n-j}}_{V}
=\norm{\sum_{j=0}^{n-1}\tilde{M}^{\alpha}e^{jhM}\delta_{n-j}}\\
\leq & \norm{\sum_{j=0}^{n-1}\tilde{M}^{\alpha}e^{jhM}h(\varphi_1(hM)-I)\hat{f}(t_{n-j-1})}
+\norm{\sum_{j=0}^{n-1}\tilde{M}^{\alpha}e^{jhM}\int_{0}^{h}e^{(h-\tau)M}\int_{0}^{\tau}\hat{f}'(t_{n-j-1}+\sigma)d\sigma d\tau}.
\end{aligned}
\end{equation}
We  note that there exists a bounded operator $\tilde{\varphi}(hM)$ with
$$\varphi_1(hM)-I=\varphi_1(hM)-\varphi_1(0)=hM\tilde{\varphi}(hM).$$
Then combining with the Lemma 2 of \cite{Hochbruck2005b} yields
\begin{equation*}\label{e10}
\begin{aligned}
& \norm{\sum_{j=0}^{n-1}\tilde{M}^{\alpha}e^{jhM}h(\varphi_1(hM)-I)\hat{f}(t_{n-j-1})}
= h \norm{\sum_{j=0}^{n-1}hM\tilde{\varphi}(hM) e^{jhM}\cdot \tilde{M}^{\alpha}\hat{f}(t_{n-j-1})}\\
\leq & h \norm{W_{n-1}}\norm{v_1}+h\sum_{j=0}^{n-2}\norm{W_{j}}\norm{v_{n-j-1}-v_{n-j}},
\end{aligned}
\end{equation*}
where $\omega_j=hM\tilde{\varphi}(hM) e^{jhM},
v_j=\tilde{M}^{\alpha}\hat{f}(t_{j-1})$ and
$W_k=\sum_{j=0}^{k}\omega_j$. With the help of Lemma 3.1 of
\cite{Hochbruck2005a}, we have $\norm{W_{j}} \leq C$. Furthermore,
on the basis of the  facts that $\norm{f}_{V}$ and $\norm{f'}_{V}$
are bounded, we then get
\begin{equation}\label{e11}
\begin{aligned}
 \norm{\sum_{j=0}^{n-1}\tilde{M}^{\alpha}e^{jhM}h(\varphi_1(hM)-I)\hat{f}(t_{n-j-1})}
\leq  Ch.
\end{aligned}
\end{equation}
Similarly to the above analysis, it follows that
\begin{equation*}\label{e12}
\begin{aligned}
&\norm{\sum_{j=0}^{n-1}\tilde{M}^{\alpha}e^{jhM}\int_{0}^{h}e^{(h-\tau)M}\int_{0}^{\tau}\hat{f}'(t_{j-1}+\sigma)d\sigma d\tau}\\
\leq &\int_{0}^{h}\norm{\sum_{j=0}^{n-1} h\tilde{M}e^{(jh+h-\tau)M}\cdot \frac{1}{h}\int_{0}^{\tau}\tilde{M}^{\alpha-1}\hat{f}'(t_{n-j-1}+\sigma)d\sigma}  \\
\leq & \int_{0}^{h} \Big( \norm{W_{n-1}}\norm{v_1}+\sum_{j=0}^{n-2}\norm{W_{j}}\norm{v_{n-j-1}-v_{n-j}}\Big)d\tau,
\end{aligned}
\end{equation*}
where $\omega_j=h\tilde{M}e^{(jh+h-\tau)M}$ and $v_j=\frac{1}{h}\int_{0}^{\tau}\tilde{M}^{\alpha-1}\hat{f}'(t_{j-1}+\sigma)d\sigma$. Using the expression of $v_j$ and conditions $\tilde{M}^{\alpha-1}f^{r}\in L^{\infty}(0,T;V), r=1,2$, we obtain
\begin{equation*}\label{e13}
\begin{aligned}
\norm{v_1}=\norm{\frac{1}{h}\int_{0}^{\tau}\tilde{M}^{\alpha-1}\hat{f}'(t_{0}+\sigma)d\sigma}\leq C,
\end{aligned}
\end{equation*}
and
\begin{equation*}\label{e14}
\begin{aligned}
\norm{v_{n-j-1}-v_{n-j}}=\norm{\frac{1}{h}\int_{0}^{\tau}\tilde{M}^{\alpha-1}\Big(\hat{f}'(t_{n-j-2}+\sigma)-\hat{f}'(t_{n-j-1}+\sigma)\Big)d\sigma}
\leq C.
\end{aligned}
\end{equation*}
Then combining with $\norm{W_{j}} \leq C$ gives
\begin{equation}\label{e15}
\begin{aligned}
\norm{\sum_{j=0}^{n-1}\tilde{M}^{\alpha}e^{jhM}\int_{0}^{h}e^{(h-\tau)M}\int_{0}^{\tau}\hat{f}'(t_{j-1}+\sigma)d\sigma d\tau}
\leq Ch.
\end{aligned}
\end{equation}
Inserting the {formulas} \eqref{e11} and \eqref{e15} into \eqref{e9}
yields
\begin{equation*}
\begin{aligned}
\norm{\sum_{j=0}^{n-1}e^{jhM}\delta_{n-j}}_{V}\leq Ch.
\end{aligned}
\end{equation*}
Therefore, according to the above analysis, we obtain
\begin{equation*}
\begin{aligned}
\norm{e_{n}}_{V}\leq CLh\sum_{j=0}^{n-1}t_{n-j-1}^{\alpha}\norm{e_{j}}_{V}+ Ch .
\end{aligned}
\end{equation*}
Finally, using the Gronwall's inequality, we arrive at the final
conclusion as follows
\begin{equation*}
\begin{aligned}
\norm{e_{n}}_{V}\leq Ch .
\end{aligned}
\end{equation*}
The proof of the theorem is complete. \hfill  $\square$
\end{proof}

\section{Numerical experiments}\label{numExperiments}

Since it is well known that explicit exponential integrators
outperform standard integrators, we do not consider standard
integrators in our numerical experiments. Differently from multistep
methods, a significant advantage of one-step methods is conceptually
simple and easy to change stepsize. Therefore, it seems plausible
that our numerical experiments are implemented under the assumption
that the variable stepsize is allowed at each time step.

In this section, we carry out numerical experiments to show the high
efficiency  of our methods. We select the following methods to make
comparisons:
\begin{itemize}   \item First-order methods: \begin{itemize}
  \item E-Euler: the explicit exponential
Euler method \eqref{Pro-exponential-Euler} of order one proposed in \cite{Hochbruck2010};
  \item MVERK1: the  1-stage explicit MVERK/SVERK \eqref{MVERK-one-stage-Euler}  of order one presented in this paper.
\end{itemize}
\item Second-order methods:
\begin{itemize}
  \item ERK2: the explicit exponential Runge-Kutta
 method  of order two proposed in \cite{Hochbruck2010};

  \item MVERK2-1: the  2-stage explicit MVERK  \eqref{MVERK-two-stage-1} of order two presented in this paper;
    \item MVERK2-2: the  2-stage explicit MVERK  \eqref{MVERK-two-stage-2}  of order two presented in this paper;
      \item SVERK2-1: the  2-stage explicit SVERK  \eqref{SVERK-two-stage-1}  of order two presented in this paper;
    \item SVERK2-2: the  2-stage explicit SVERK  \eqref{SVERK-two-stage-2}  of order two presented in this paper.
\end{itemize}

\item Third-order methods:
\begin{itemize}
  \item ERK3: the explicit exponential Runge-Kutta
 method  of order three proposed in \cite{Hochbruck2010};
  \item MVERK3-1: the  3-stage explicit MVERK  \eqref{MVERK-three-stage-1} of order three presented in this paper;
    \item MVERK3-2: the  3-stage explicit MVERK  \eqref{MVERK-three-stage-2}  of order three presented in this paper;
      \item SVERK3-1: the  3-stage explicit SVERK  \eqref{SVERK-three-stage-1}  of order three presented in this paper;
    \item SVERK3-2: the  3-stage explicit SVERK \eqref{SVERK-three-stage-2}  of order three presented in this paper.
\end{itemize}
\end{itemize}

\begin{problem}\label{Allen-Cahn}
We first consider a stiff partial differential equation: Allen--Cahn
equation. Allen--Cahn equation (see, e.g. \cite{Cox2002,Kassam2005})
is a reaction-diffusion equation of mathematical physics,  given by
$$u_{t}-\epsilon u_{xx}=u-u^3,\ \ x\in[-1,1],$$
with   $\epsilon= 0.01$ and initial conditions
$$ u(x,0)=0.53x+0.47\sin(-1.5\pi x),\ \ u(-1,t)=-1,\ \ u(1,t)=1.$$
 We use  a 32-point Chebyshev spectral method which yields a system
of ordinary differential equations
$$U_{t}-AU=U-U^3.$$
 We
apply the MATLAB function \emph{cheb} from \cite{Trefethen2000} for
the grid generation and obtain the differentiation matrix $M$. { The
form of this Chebyshev differentiation matrix is referred to Theorem
7 of Chapter 6 in \cite{Trefethen2000}.} It is noted that the
differentiation matrix $M$ in this example is full. This system is
integrated on $[0,1]$ with different stepsizes $h=1/2^k$ for
$k=8,9,\ldots,13$. The  global errors $GE:= \norm{U_{n}-U(t_n)}$
against the stepsizes and the CPU time are given  in Figs.
\ref{PP1-1} and \ref{PP1-2}, respectively.

\begin{figure}[!htb]
\centering
\begin{tabular}[c]{cccc}%
  \subfigure[]{\includegraphics[width=4.5cm,height=4.5cm]{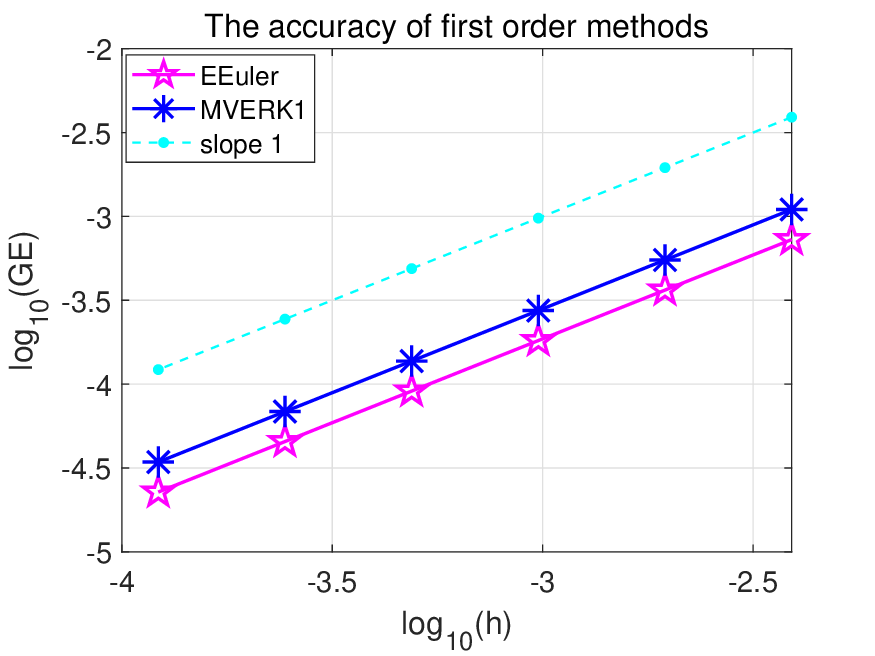}}
  \subfigure[]{\includegraphics[width=4.5cm,height=4.5cm]{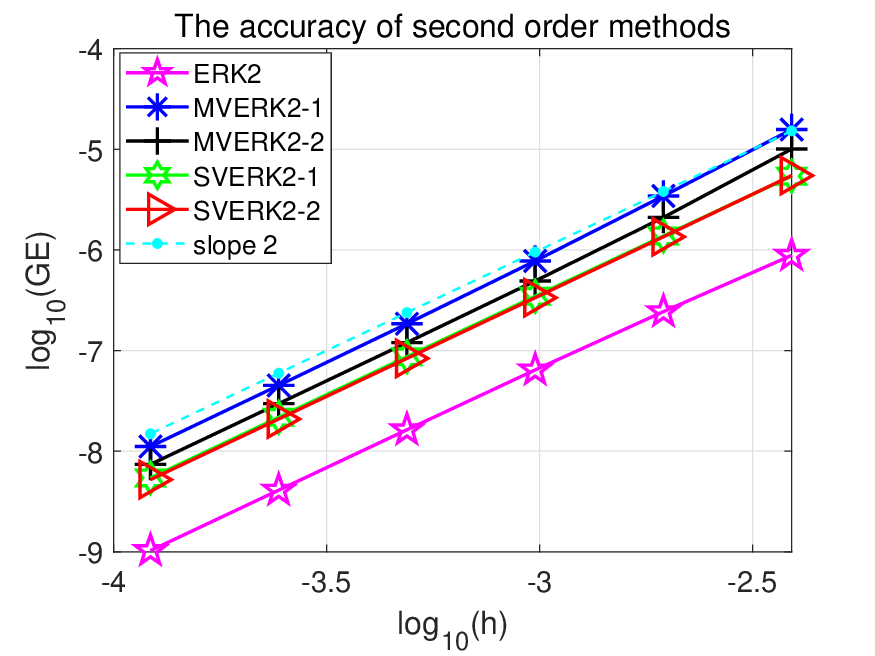}}
    \subfigure[]{\includegraphics[width=4.5cm,height=4.5cm]{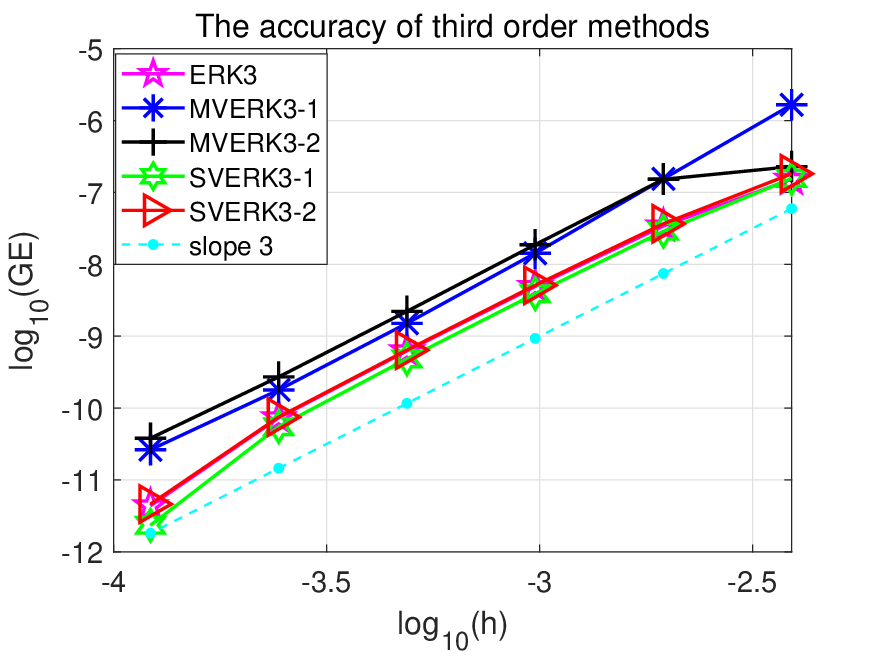}}
\end{tabular}
\caption{Results for accuracy of Problem 1: The $\log$-$\log$ plots of global
errors against $h$}.\label{PP1-1}
\end{figure}

\begin{figure}[!htb]
\centering
\begin{tabular}[c]{cccc}%
  \subfigure[]{\includegraphics[width=4.5cm,height=4.5cm]{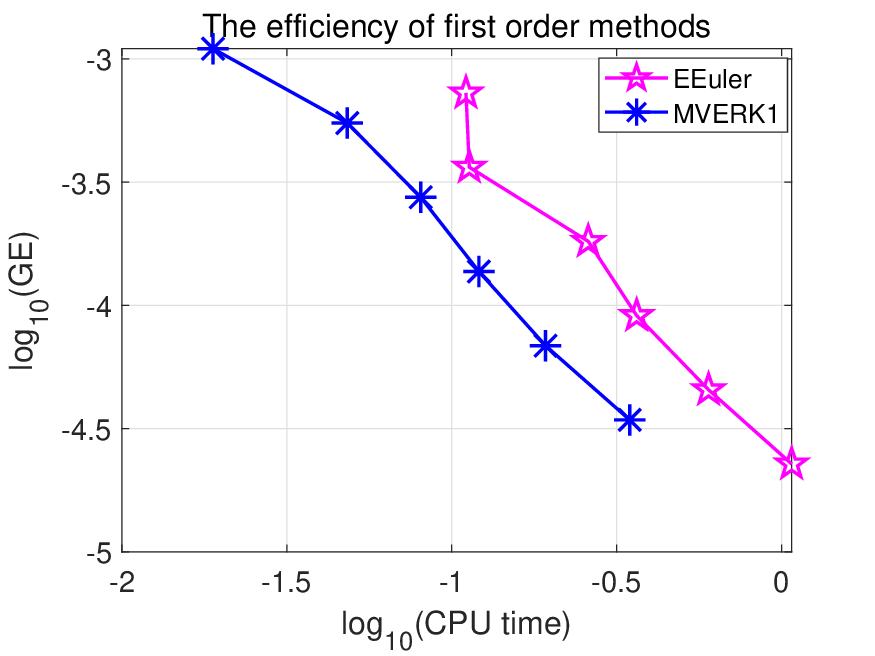}}
  \subfigure[]{\includegraphics[width=4.5cm,height=4.5cm]{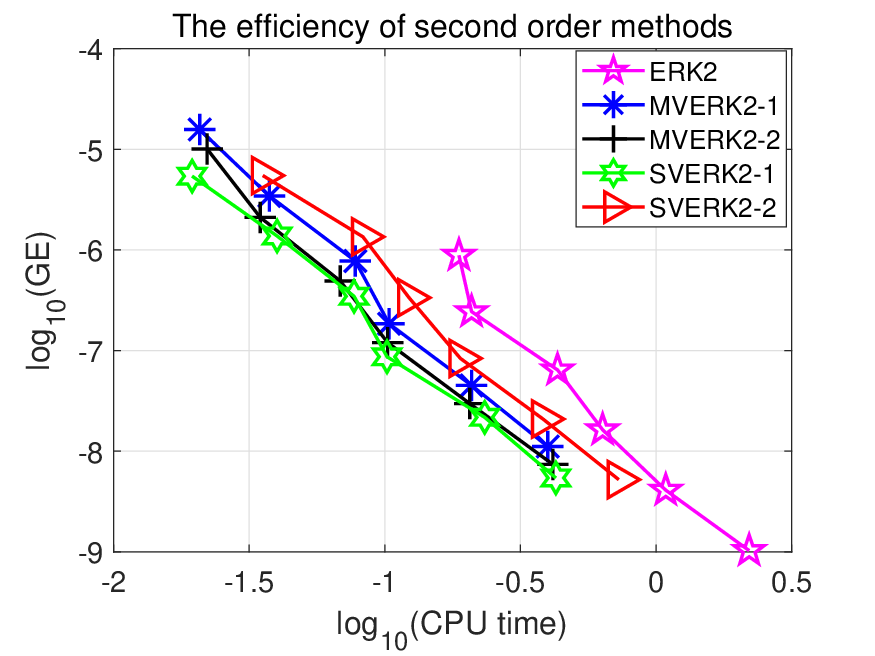}}
    \subfigure[]{\includegraphics[width=4.5cm,height=4.5cm]{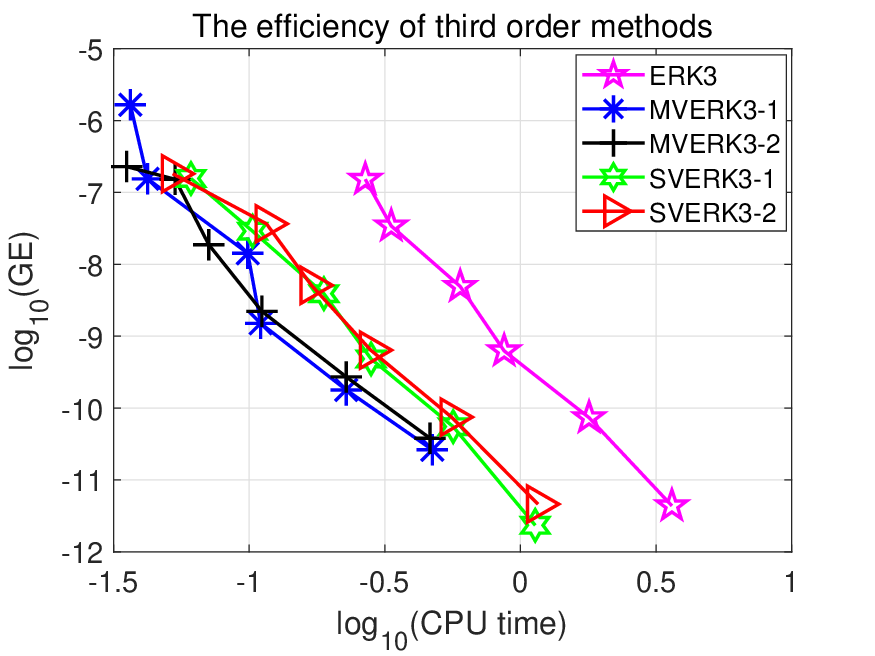}}
\end{tabular}
\caption{Results for efficiency of Problem 1: The $\log$-$\log$
plots of global errors against the CPU time.}\label{PP1-2}
\end{figure}
\end{problem}

\begin{problem}\label{Wined-induced-oscillation}
Consider the following averaged system in wind-induced
 oscillation (see
 \cite{Mclachlan-98})
\begin{equation*}
\begin{aligned}& \left(
                   \begin{array}{c}
                     x_1 \\
                      x_2 \\
                   \end{array}
                 \right)
'= \left(
    \begin{array}{cc}
      -\zeta& -\lambda\\
       \lambda & -\zeta \\
    \end{array}
  \right)\left(
                   \begin{array}{c}
                     x_1 \\
                      x_2 \\
                   \end{array}
                 \right)+
\left(
                                                                           \begin{array}{c}
                                                                           x_1x_2\\
\frac{1}{2}(x_1^2-x_2^2)
                                                                           \end{array}
                                                                         \right),
\end{aligned}\end{equation*}
 where $\zeta \geq
0$ is a damping factor and $\lambda$ is a detuning parameter. We
solve this system on $[0,10]$ with   $h=1/2^k$ for $k=3,4,\ldots,8$.
Figs. \ref{PP3-1} and \ref{PP3-2} display  the global errors $GE$
against the stepsizes and the CPU time, respectively.

\begin{figure}[!htb]
\centering
\begin{tabular}[c]{cccc}%
  \subfigure[]{\includegraphics[width=4.5cm,height=4.5cm]{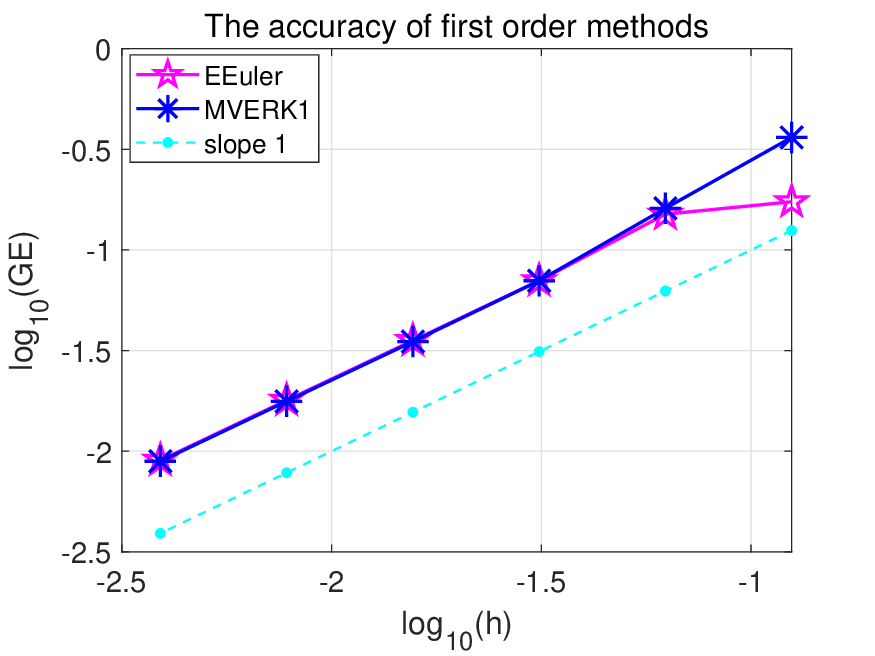}}
  \subfigure[]{\includegraphics[width=4.5cm,height=4.5cm]{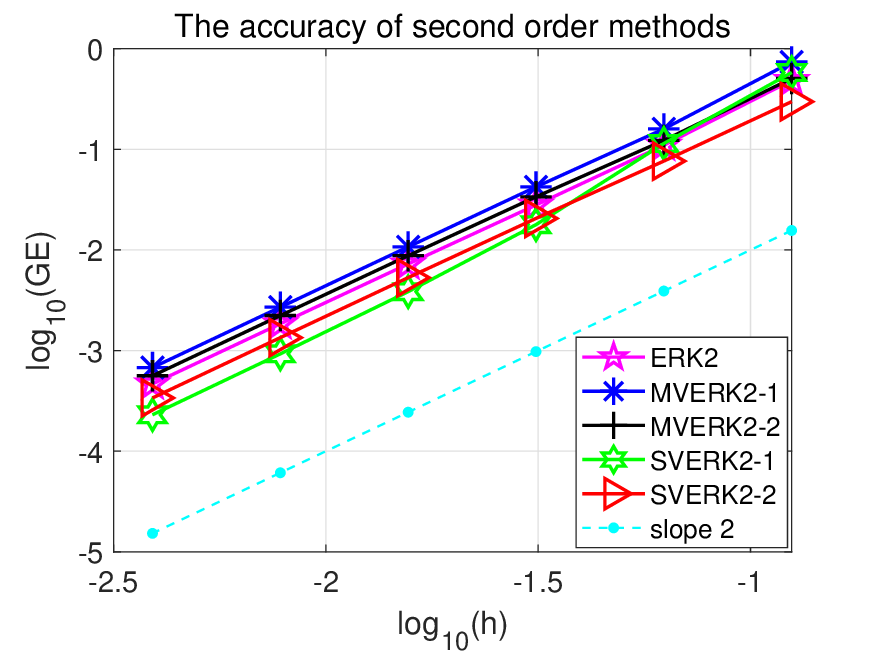}}
    \subfigure[]{\includegraphics[width=4.5cm,height=4.5cm]{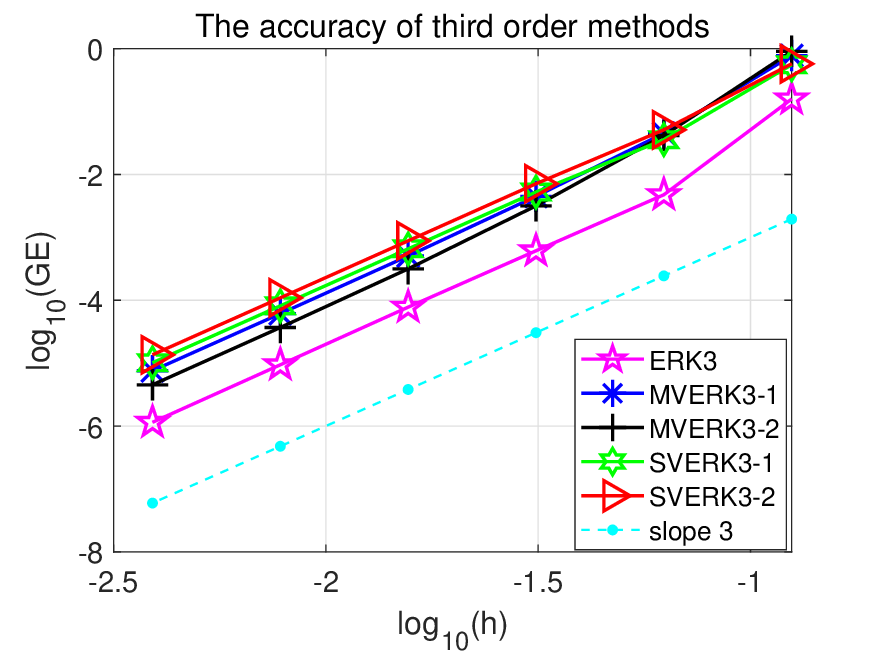}}
\end{tabular}
\caption{Results for accuracy of Problem 2: The $\log$-$\log$ plots
of global errors against $h$.}\label{PP3-1}
\end{figure}

\begin{figure}[!htb]
\centering
\begin{tabular}[c]{cccc}%
  \subfigure[]{\includegraphics[width=4.5cm,height=4.5cm]{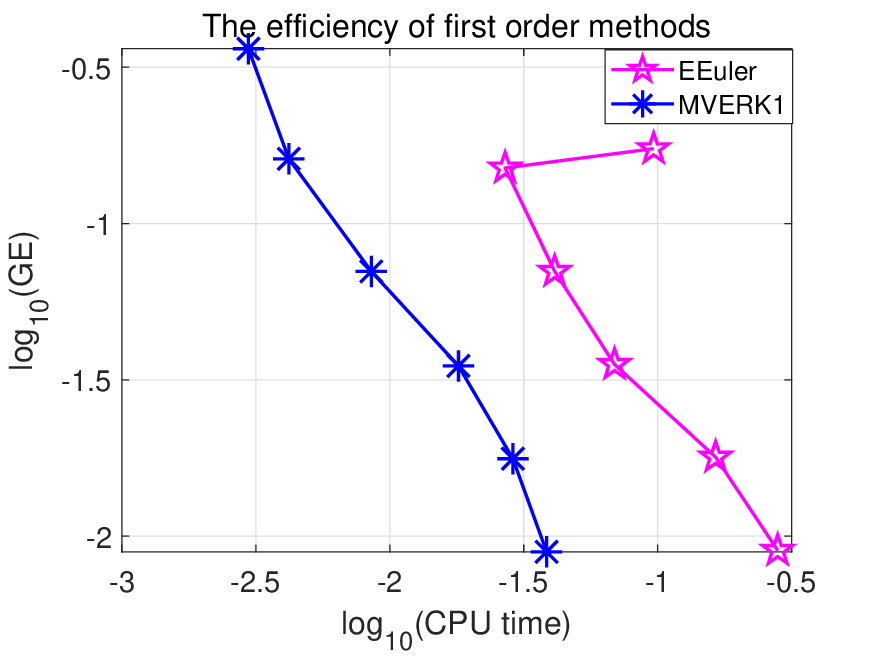}}
  \subfigure[]{\includegraphics[width=4.5cm,height=4.5cm]{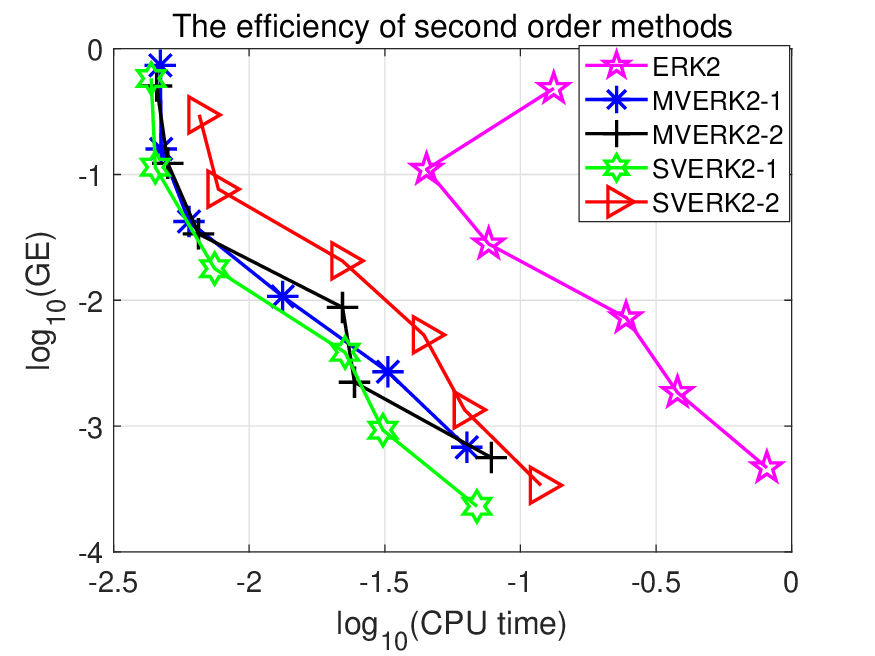}}
    \subfigure[]{\includegraphics[width=4.5cm,height=4.5cm]{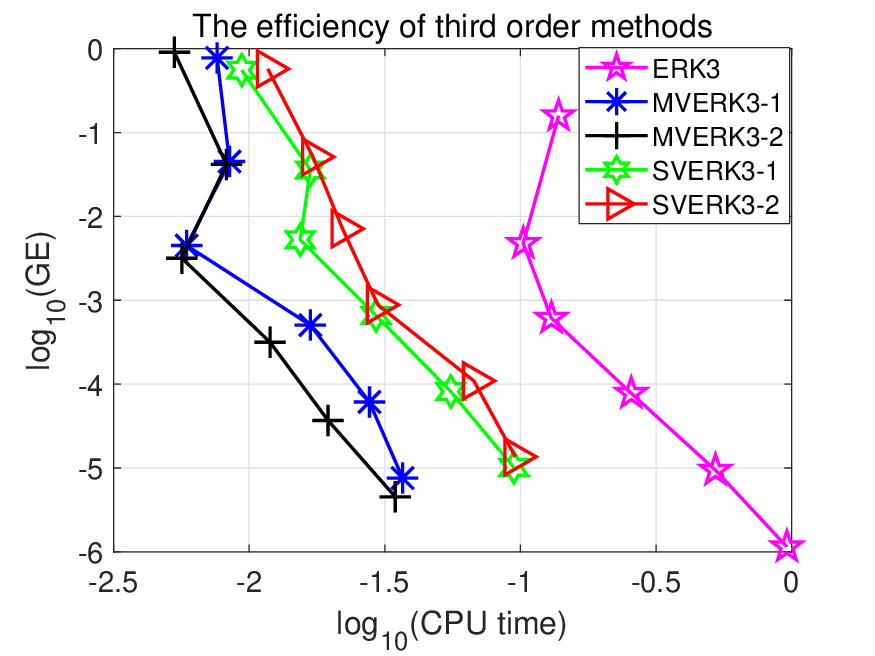}}
\end{tabular}
\caption{Results for efficiency of Problem 2: The $\log$-$\log$
plots of global errors against the CPU time.}\label{PP3-2}
\end{figure}
\end{problem}

\begin{problem}\label{Non-Schrodinger}
 Consider  the nonlinear
Schr\"{o}dinger equation (see  \cite{Chen2001})
\begin{equation*}\begin{aligned}
&i\psi_{t}+\psi_{xx}+2|\psi|^{2}\psi=0,\quad
 \psi(x,0)= 0.5 + 0.025 \cos(\mu x),\\
\end{aligned}
\end{equation*}
with the periodic boundary condition $\psi(0,t)=\psi(L,t).$
Following \cite{Chen2001},  we choose $L =4\sqrt{2}\pi$ and $\mu =
2\pi/L.$ The initial condition chosen here is in the vicinity of the
homoclinic orbit.
\end{problem}

Using $\psi = p + \textmd{i}q,$ this equation can be rewritten as a
pair of real-valued equations
\begin{equation*}\label{semi}
\begin{aligned}
&p_t +q_{xx} + 2(p^2  + q^2)q = 0,\\
&q_t -p_{xx} -2(p^2  + q^2)p = 0.\\
\end{aligned}
\end{equation*}
Discretising the spatial derivative $\partial_{xx}$ by the
pseudospectral method given in  \cite{Chen2001},  this problem is
converted into the following system:
\begin{equation}\label{semi sd}
\left(
  \begin{array}{c}
    \textbf{p} \\
    \textbf{q} \\
  \end{array}
\right)'=
 \begin{aligned}
 \left(
   \begin{array}{cc}
     0 & -D_2 \\
     D_2 & 0 \\
   \end{array}
 \right)\left(
  \begin{array}{c}
    \textbf{p} \\
    \textbf{q} \\
  \end{array}
\right)+\left(
          \begin{array}{c}
            -2(\textbf{p}^{2}+ \textbf{q}^{2})\cdot  \textbf{q} \\
            2( \textbf{p}^{2}+ \textbf{q}^{2})\cdot \textbf{p} \\
          \end{array}
        \right)
\end{aligned}
\end{equation}
where $\textbf{p}=(p_0,p_1,\ldots,p_{N-1})^{\intercal},\
\textbf{q}=(q_0,q_1,\ldots,q_{N-1})^{\intercal}$ and
$D_2=(D_2)_{0\leq j,k\leq N-1}$ is the pseudospectral differential
matrix defined by:
\begin{equation*}
(D_2)_{jk}=\left\{\begin{aligned}
&\frac{1}{2}\mu^2(-1)^{j+k+1}\frac{1}{\sin^2(\mu(x_j-x_k)/2)},\quad j\neq k,\\
&-\mu^2\frac{2(N/2)^2+1}{6},\quad\quad\quad\quad\quad\quad\ \ \ \ \   j=k.
\end{aligned}\right.
\end{equation*}
In this test, we choose $N=64$  and integrate the
system on $[0,1]$ with $h=1/2^k$ for $k=2,3,\ldots,7$. The  global
errors $GE$  against the stepsizes and the CPU time are respectively
presented  in Figs. \ref{PP4-1} and \ref{PP4-2}.

From the results of these three numerical experiments, we have the
following observations. Although the  MVERK and SVERK
methods derived in this paper have comparable accuracy in comparison
with standard exponential integrators, our exponential methods
demonstrate lower computational cost and more competitive
efficiency.
\begin{figure}[!htb]
\centering
\begin{tabular}[c]{cccc}%
  \subfigure[]{\includegraphics[width=4.5cm,height=4.5cm]{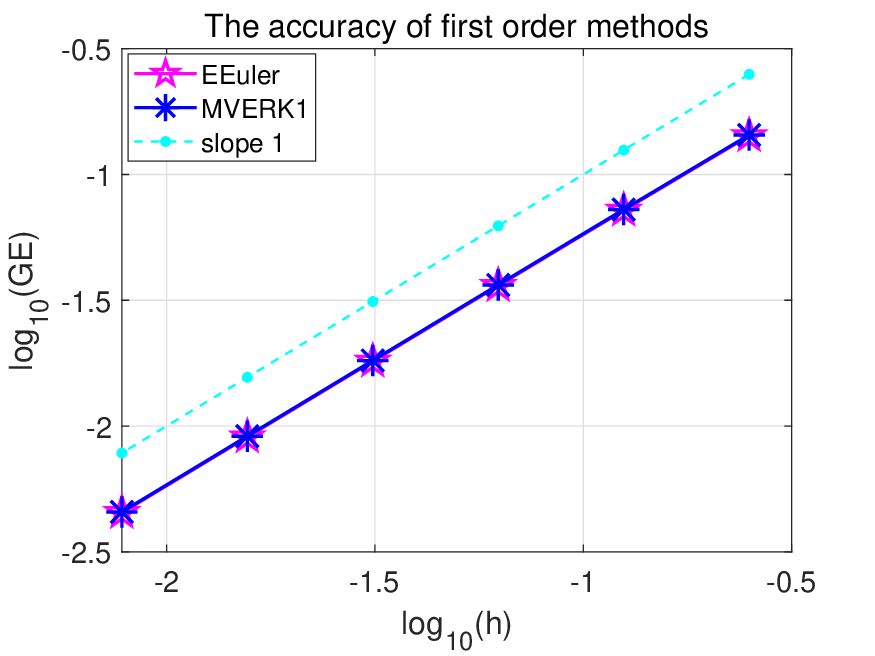}}
  \subfigure[]{\includegraphics[width=4.5cm,height=4.5cm]{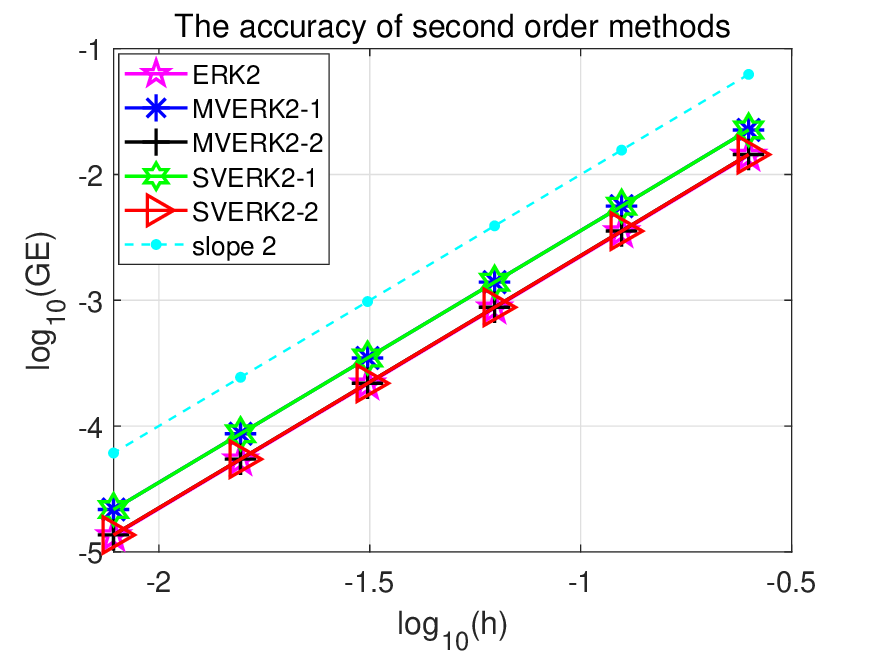}}
    \subfigure[]{\includegraphics[width=4.5cm,height=4.5cm]{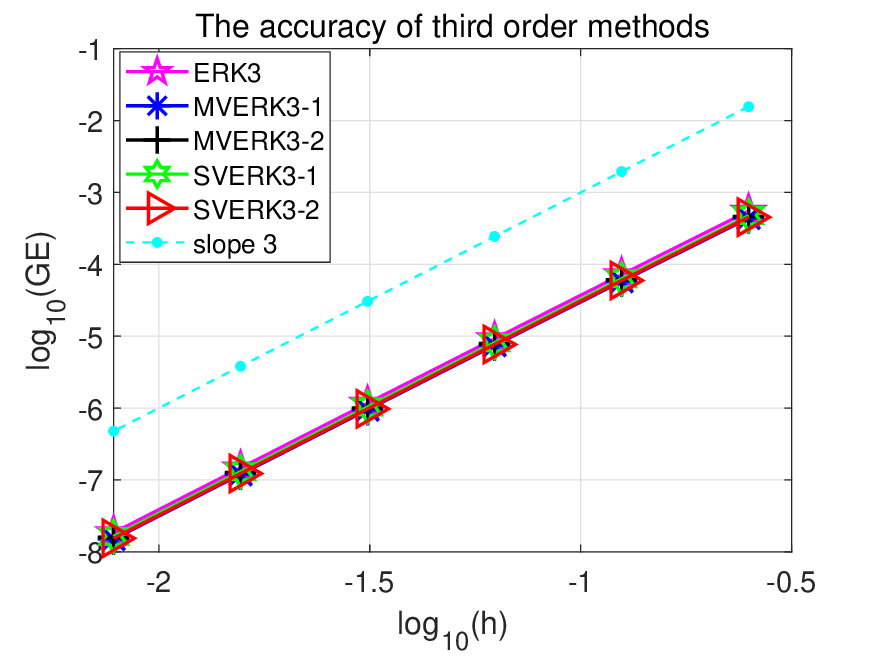}}
\end{tabular}
\caption{Results for accuracy of Problem 3: The $\log$-$\log$ plots
of global errors against $h$.}\label{PP4-1}
\end{figure}

\begin{figure}[!htb]
\centering
\begin{tabular}[c]{cccc}%
  \subfigure[]{\includegraphics[width=4.5cm,height=4.5cm]{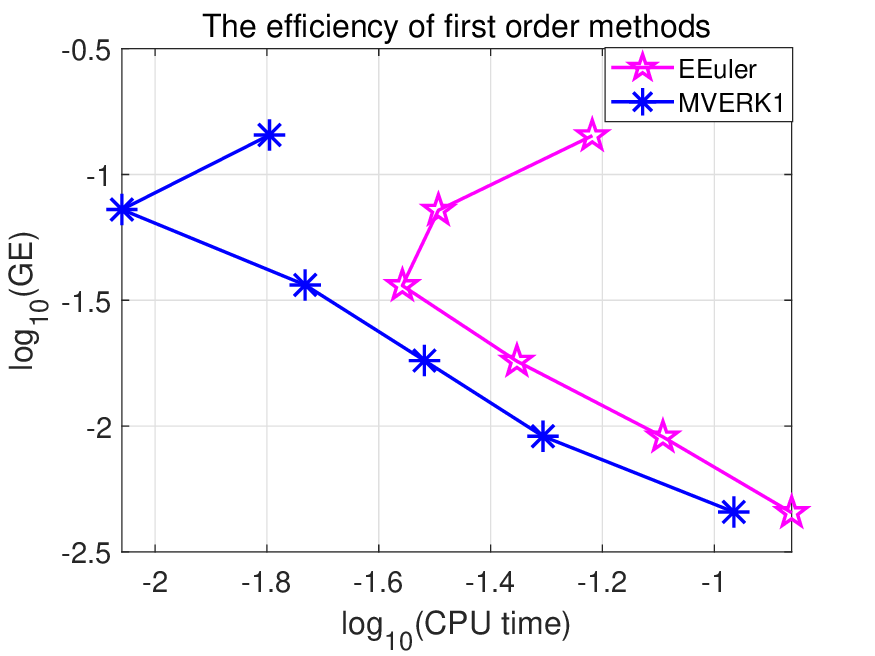}}
  \subfigure[]{\includegraphics[width=4.5cm,height=4.5cm]{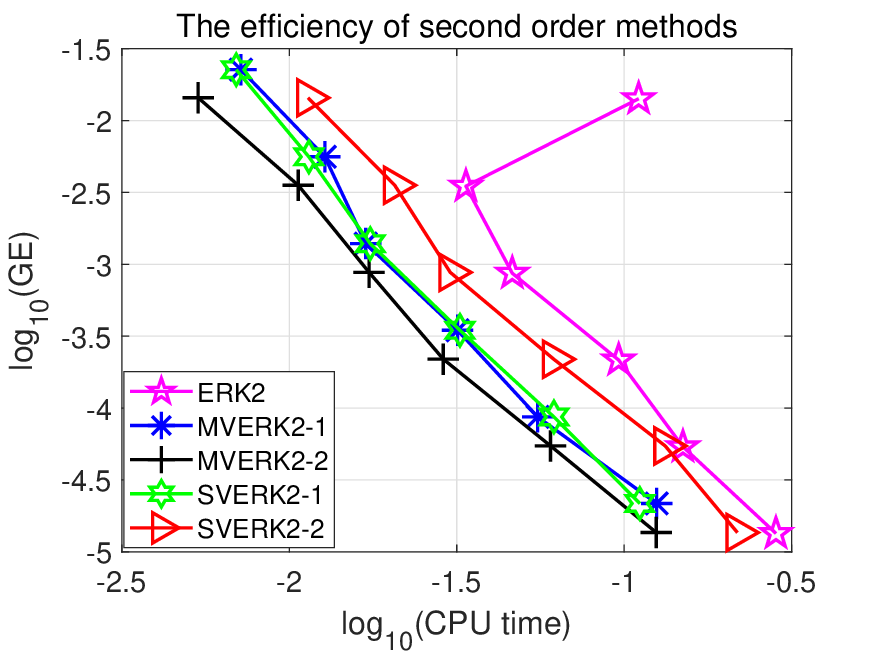}}
    \subfigure[]{\includegraphics[width=4.5cm,height=4.5cm]{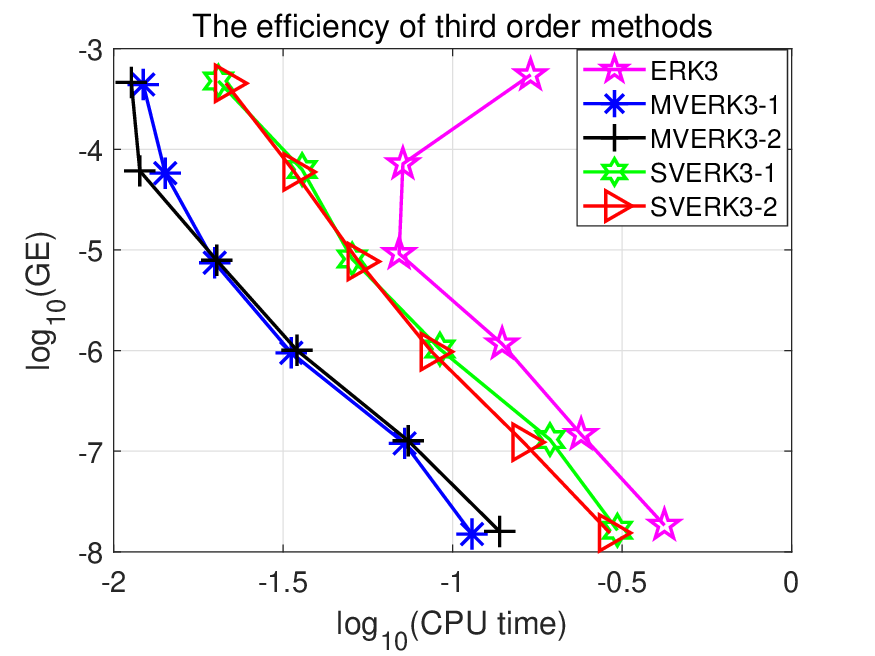}}
\end{tabular}
\caption{Results for efficiency of Problem 3: The $\log$-$\log$
plots of global errors against the CPU time.}\label{PP4-2}
\end{figure}

\section{Conclusions and further research}

As we have known, exponential integrators are very promising for
solving semi-linear systems whose linear part generates the dominant
stiffness or high oscillation of the underlying problem. In this
work, we have presented two new classes of exponential integrators
to solve  stiff systems or highly oscillatory problems. A
distinctive feature of these new integrators is the exact evaluation
of the contribution brought  by the linear term appearing in
\eqref{stiffODEs}. We mainly focus on explicit modified and
simplified exponential integrators in this paper.

A key feature of our approach is that, unlike standard exponential
integrators, our new exponential integrators reduce computational
cost brought by evaluations of matrix exponentials, whereas the
computational cost of exponential integrators appeared in the
literature is heavily dependent on the evaluations of matrix
exponentials. For instance, the numerical results of both
first-order  explicit exponential methods
\eqref{MVERK-one-stage-Euler} and \eqref{Pro-exponential-Euler} show
 almost the same accuracy. However, a closer look at the CPU time
reveals that our first-order explicit exponential method
\eqref{MVERK-one-stage-Euler} is better than the well-known
prototype exponential method \eqref{Pro-exponential-Euler}. The main
reason is that the coefficient of our method
\eqref{MVERK-one-stage-Euler} is independent of the evaluation of
exponential matrix. On the contrary, the prototype exponential
method \eqref{Pro-exponential-Euler} is dependent on the evaluation
of exponential matrix determined by \eqref{varphi1}.  For comparable
accuracy, our explicit exponential methods require in general less
work than the standard explicit exponential method in the
literature. Moreover, the larger the dimension $d$ of  matrix $M$
is, the higher computational cost will be.

We analysed the convergence of the explicit  MVERK method
\eqref{MVERK-one-stage-Euler}. An interesting conclusion
from Theorem \ref{A-Stab-23} is that  the explicit MVERK and SVERKN
methods yielded from  order conditions of standard $p$th-order  RK
methods are of order $p$ for $p=1,2,3$,  when applied to
\eqref{stiffODEs}. Moreover, all of them are $A$-stable in the sense
of Dahlquist. 

Finally, we  presented numerical examples based on Allen--Cahn
equation, the averaged system in wind-induced
 oscillation and the nonlinear Schr\"{o}dinger equation,
which are very relevant in applications. It follows from
the numerical results that our new explicit ERK integrators have
higher
efficiency than the standard ERK methods.

The implicit MVERK integrators and SVERK integrators can be further
investigated.


%


\begin{thebibliography}{44}


\bibitem{Berland2005}
H. Berland, B. Owren, B. Skaflestad,
B-series and order conditions for exponential integrators,
SIAM J. Numer. Anal. 43 (2005) 1715-1727.
\bibitem{Berland2007Package}
 H. Berland, B. Skaflestad,
W. Wright, EXPINT--A MATLAB package for exponential integrators, ACM
Trans. Math. Softw. 33 (2007) 4.


\bibitem{JCP(1998)_Beylkin}G. Beylkin, J. M. Kelser, L.
Vozovol, A new class of time discretization schemes for the solution
of nonlinear PDEs, J. Comput. Phys. 147 (1998) 362-387.

\bibitem{Bhatty2017}
A. Bhatt, B.E. Moore, Structure-preserving exponential Runge-Kutta
methods, SIAM J. Sci. Comput. 39 (2017) A593-A612.

\bibitem{Butcherbook2008}J.C. Butcher, Numerical Methods for
Ordinary Diffrential Equations, Second edition, John Wiley \& Sons,
Ltd, 2008.







\bibitem{Celledoni2008}
E. Celledoni, D. Cohen, B. Owren,
Symmetric exponential integrators with an application to the cubic
Schr\"{o}inger equation, Found. Comput. Math. 8 (2008) 303-317.

\bibitem{Chen2001} J. B. Chen, M. Z. Qin, {Multisymplectic Fourier pseudospectral method for the nonlinear
Schr\"{o}dinger equation}, Electron. Trans. Numer. Anal. 12 (2001)
193-204.



\bibitem{Cox2002}
S. Cox, P. Matthews,
Exponential time differencing for stiff systems,
J. Comput. Phys. 176 (2002) 430-455.

 \bibitem{DahlquistBIT1963} G. G. Dahlquist, A special stability
problem for linear multistep methods, BIT 3 (1963) 27-43.




\bibitem{Dimarco2011}
G. Dimarco, L. Pareschi,
Exponential Runge-Kutta methods for stiff kinetic equations,
SIAM J. Numer. Anal. 49 (2011) 2057-2077.

\bibitem{Dujardin2009}
G. Dujardin, Exponential Runge-Kutta methods for the Schr\"{o}inger
equation, Appl. Numer. Math. 59 (2009) 1839-1857.



\bibitem{Hochbruck1997}
M. Hochbruck, C. Lubich, On Krylov subspace approximations to
the matrix exponential operator, SIAM J. Numer. Anal.  34 (1997) 1911-1925.

\bibitem{Hochbruck1998}
M. Hochbruck, C. Lubich, H. Selhofer,
Exponential integrators for large systems of differential equations,
SIAM J. Sci. Comput. 19 (1998) 1552-1574.


\bibitem{Hochbruck2005a}
M. Hochbruck, A. Ostermann,
Explicit exponential Runge-Kutta methods for semilinear
parabolic problems,
SIAM J. Numer. Anal. 43 (2005) 1069-1090.

\bibitem{Hochbruck2005b}
M. Hochbruck, A. Ostermann,
Exponential Runge-Kutta methods for parabolic problems,
Appl. Numer. Math. 53 (2005) 323-339.

\bibitem{Hochbruck2010}
M. Hochbruck, A. Ostermann,
Exponential integrators,
Acta Numer. 19 (2010) 209-286.



\bibitem{Kassam2005}
A.-K. Kassam, L.N. Trefethen,
Fourth-order time stepping for stiff PDEs,
SIAM J. Sci. Comput. 26 (2005) 1214-1233.



\bibitem{Lawson1967}
J.D. Lawson, Generalized Runge-Kutta processes for stable systems
with large Lipschitz constants, SIAM J. Numer. Anal. 4 (1967)
372-380.

\bibitem{LiYW2016a} Y.W. Li, X. Wu,
Exponential integrators
preserving first integrals or Lyapunov functions for conservative or
dissipative systems, SIAM J. Sci. Comput. 38 (2016) A1876-A1895.

\bibitem{Mclachlan-98}
R. I. Mclachlan, G. R. W. Quispel, N. Robidoux, {A unified approach
to Hamiltonian systems, Poisson systems, gradient systems, and
systems with Lyapunov functions or first integrals}, Phys. Rev.
Lett.  81 (1998) 2399-2411.





\bibitem{Mei2017}
L. Mei, X. Wu,
Symplectic exponential Runge-Kutta methods for solving nonlinear
Hamiltonian systems, J. Comput. Phys. 338 (2017) 567-584.

\bibitem{Mei-Siam2022}
L. Mei, L. Huang, X. Wu, Energy-preserving continuous-stage exponential Runge--Kutta integrators for efficiently solving Hamiltonian systems, SIAM J. Sci. Comput. 44 (2022) A1092-A1115.





\bibitem{Shen2019}
X. Shen, M. Leok,
Geometric exponential integrators, J. Comput. Phys.
382 (2019) 27-42.




\bibitem{Trefethen2000}
L.N. Trefethen,
Spectral methods in MATLAB, SIAM, Philadelphia,
2000.



\bibitem{WangBinJCAM2019} B. Wang, X. Wu, Exponential collocation
methods for conservative or dissipative systems, J. Comput. Appl.
Math. 360 (2019) 99-116.

\bibitem{WangBinJCP2019}B. Wang, X.  Wu,  Volume-preserving exponential integrators and their applications, J. Comput. Phys. 396 (2019) 867-887.

\bibitem{WWbook2021}B. Wang, X.  Wu, Geometric Integrators for Differential Equations with Highly
Oscillatory Solutions, Springer Nature Singapore Pte Ltd. 2021

\bibitem{WangBin2021}
B. Wang, X. Zhao,
Error estimates of some splitting schemes for
charged-particle dynamics under strong magnetic field,
SIAM J. Numer. Anal. 59 (2021)  2075-2105.


\end{thebibliography}
\end{document}